\documentclass[12pt,3p]{elsarticle}
\usepackage{amsmath, amsfonts}
\usepackage{amsthm}
\usepackage{enumitem}
\usepackage{xcolor}
\usepackage{mathtools}
\usepackage{amssymb}
\usepackage{hyperref}
\usepackage{etoolbox}
\usepackage{tgtermes}
\DeclareMathOperator{\Trig}{Trig}
\DeclareMathOperator{\dive}{div}
\DeclareMathOperator{\loc}{loc}
\DeclareMathOperator{\Id}{Id}
\patchcmd{\thmhead}{(#3)}{#3}{}{}
\usepackage{subcaption}
\usepackage{graphicx}
\usepackage{natbib}
\numberwithin{equation}{section}
\makeatletter
\def\ps@pprintTitle{%
	\let\@oddhead\@empty
	\let\@evenhead\@empty
	\def\@oddfoot{}%
	\let\@evenfoot\@oddfoot}
\makeatother
\def\Xint#1{\mathchoice
	{\XXint\displaystyle\textstyle{#1}}%
	{\XXint\textstyle\scriptstyle{#1}}%
	{\XXint\scriptstyle\scriptscriptstyle{#1}}%
	{\XXint\scriptscriptstyle\scriptscriptstyle{#1}}%
	\!\int}
\def\XXint#1#2#3{{\setbox0=\hbox{$#1{#2#3}{\int}$ }
		\vcenter{\hbox{$#2#3$ }}\kern-.6\wd0}}

\def\dashint{\Xint-}
\newtheorem{theorem}{Theorem}[section]
\newtheorem{lemma}[theorem]{Lemma}
\newtheorem{proposition}[theorem]{Proposition}

\theoremstyle{definition}
\newtheorem{definition}[theorem]{Definition}

\theoremstyle{remark}
\newtheorem{remark}[theorem]{Remark}

\usepackage{graphicx}
\begin{document}
	
	\begin{frontmatter}
		
		
		
		
		\title{Bloch wave approach to almost periodic homogenization and approximations of effective coefficients}		
		\author{Sivaji Ganesh Sista\corref{cor1}}
		\ead{sivaji.ganesh@iitb.ac.in}
		
		\author{Vivek Tewary\corref{cor2}}
		\ead{vivekt@iitb.ac.in}
		\cortext[cor2]{Corresponding author}
		\address{Department of Mathematics, Indian Institute of Technology Bombay, Powai, Mumbai, 400076, India.}
		\begin{abstract}
			Bloch wave homogenization is a spectral method for obtaining effective coefficients for periodically heterogeneous media. This method hinges on the direct integral decomposition of periodic operators, which is not available in a suitable form for almost periodic operators. In particular, the notion of Bloch eigenvalues and eigenvectors does not exist for almost periodic operators. However, we are able to recover the almost periodic homogenization result by employing a sequence of periodic approximations to almost periodic operators. We also establish a rate of convergence for approximations of homogenized tensors for a class of almost periodic media. The results are supported by a numerical study.
		\end{abstract}
		
		\begin{keyword}
			Bloch eigenvalues \sep Almost Periodic Operators \sep Homogenization
			\MSC[2010] 47A55 \sep 35J15 \sep 35B27 \sep 34C27
		\end{keyword}
	
	\end{frontmatter}

\section{Introduction}
The aim of this paper is to extend the framework of Bloch wave method~\cite{Conca1997} to almost periodic media. Many microstructures beyond periodic occur in nature, such as amorphous solids like glass, motion of 2D electrons in a magnetic field~\cite{Hofstadter76}, quasicrystals~\cite{Schechtman84}, etc. The mixing together of two periodic media or an interface problem involving two different periodic media on the two sides of the interface~\cite{Blanc2015} may be thought of as an almost periodic microstructure. Quasicrystals, which were discovered by Schechtman~\cite{Schechtman84}, are an example of almost periodic media and are often modeled by taking projections of periodic media in higher dimensions~\cite{Katz1986}. Finally, dimers and polymers have also been modeled with almost periodic potentials~\cite{Carvalho2002}. Although almost periodic media is completely deterministic, it serves as a bridge to stochastic descriptions of nature. A large variety of seemingly random natural phenomena can be explained through almost periodic structures~\cite{Allais1983}.

The first author to study the homogenization of highly oscillatory almost periodic media was Kozlov~\cite{Kozlov78}. Unlike periodic media, the cell problem for almost periodic media is posed on $\mathbb{R}^d$ and may not have solutions in the class of almost periodic functions. This was remedied by an abstract approach outlined in~\cite{Oleinik1982,Jikov1994} where solutions to the corrector equation were sought without derivatives.

Bloch wave method of homogenization relies on direct integral decomposition of periodic operators. For almost periodic operators, a direct integral decomposition is proposed in~\cite{bellissard1981almost}, however its fibers do not have compact resolvent which prevents us from defining Bloch eigenvalues for the almost periodic operator. To overcome this difficulty, we make use of periodic approximations, which are defined by a ``restrict and periodize" operation, employed earlier by Bourgeat and Piatnitski~\cite{BourgeatPiatnitski2004} for stochastic homogenization.

Bloch wave method is a spectral method of homogenization. In particular, it relies on tools from representation theory for periodic operators~\cite{maurin68}. For definiteness, consider an operator in $L^2(\mathbb{R}^d)$ of the form 

\begin{equation}
\mathcal{F}^\epsilon u:=-\frac{\partial}{\partial x_k}\left(\kappa_{kl}\left(\frac{x}{\epsilon}\right)\frac{\partial u}{\partial x_l}\right),\label{eq1:operator}
\end{equation} where the coefficients are measurable bounded, periodic and symmetric. Let $\mathbb{T}^d$ denote the $d$-dimensional torus. Then the operator $\mathcal{F}^\epsilon$ is unitarily equivalent to a direct integral, given by 

\begin{equation}
\int^{\bigoplus}_{\mathbb{T}^d/\epsilon}\mathcal{F}^\epsilon(\xi)d\xi,
\label{eq2:representation}
\end{equation}

where the fibers  $\mathcal{F}^\epsilon(\xi)$ have compact resolvent and hence each fiber has a countable sequence of eigenvalues and eigenfunctions $\{\lambda_n^\epsilon(\xi),\phi_n^\epsilon(x,\xi)\}_{n\in\mathbb{N}}$, which are known as Bloch eigenvalues and eigenfunctions when considered as functions of $\xi\in\mathbb{T}^d/\epsilon$. Define $l^{\,th}$ Bloch coefficient of $u$ by
\begin{align*}
({\mathcal{B}}_l^\epsilon u)(\xi)=\int_{\mathbb{R}^d}\overline{{\phi}^\epsilon_l(x,\xi)}u(x)e^{-ix\cdot \xi}~dx,~l\in\mathbb{N}.
\end{align*} Then as a consequence of the representation~\eqref{eq2:representation}, the equation $\mathcal{F}^\epsilon u^\epsilon=f$, where $f\in L^2(\mathbb{R}^d)$, can be written as a cascade of equations in the Bloch space, viz.,  
\begin{align*}
\lambda^\epsilon_1(\xi)\mathcal{B}^\epsilon_1(u^\epsilon)(\xi)&=\mathcal{B}^\epsilon_1(f)\\
\lambda^\epsilon_2(\xi)\mathcal{B}^\epsilon_2(u^\epsilon)(\xi)&=\mathcal{B}^\epsilon_2(f)\\
&\vdots\\
\lambda^\epsilon_l(\xi)\mathcal{B}^\epsilon_l(u^\epsilon)(\xi)&=\mathcal{B}^\epsilon_l(f)\\
&\vdots
\end{align*}
Homogenized equation can be recovered by passing to the limit in the first equation. The rest of the equations do not contribute to homogenization. It is evident that the representation~\eqref{eq2:representation} is crucial in this method.

For almost periodic operators, we introduce periodic approximations on cubes of side length $2\pi R$ which will add yet another parameter to the problem. We perform a Bloch wave analysis of the approximation and pass to the limit in Bloch space, first as $\epsilon\to 0$, followed by $R\to\infty$. We mention some of the interesting techniques employed in this paper. The approximate Bloch spectral problems are posed on varying Hilbert spaces indexed by $R$. The approximate corrector and approximate homogenized tensors are obtained in terms of the first Bloch eigenvector and Bloch eigenvalue of the periodization. The homogenization limit is given a unified treatment by working in the Besicovitch space of almost periodic functions. We also prove a module containment result for the correctors which is of independent interest. 

Moreover we obtain a rate of convergence for the approximate homogenized tensors corresponding to periodizations for a class of almost periodic media. To this end, we use ideas from ~\cite{BourgeatPiatnitski2004},~\cite{Shen2015} and~\cite{Gloria2011}. Bourgeat and Piatnitski~\cite{BourgeatPiatnitski2004} prove a convergence rate for approximate homogenized coefficients for stochastic media under strong mixing conditions. The stochastic process generated by almost periodic media is strictly ergodic and not mixing~\cite{Simon1982}. Therefore the result of Bourgeat and Piatnitski does not apply to them. This necessitates a quantification of almost periodicity. One such quantification is proposed in~\cite{Armstrong2014}.

In a previous work~\cite{Vivek2018}, the authors consider perturbations of coefficients of an operator which make spectral edges simple. This work may also be thought of in the same vein. The almost periodic operator is expected to have a Cantor like spectrum~\cite{Damanik2019}, and hence ill-defined spectral edges. Periodic approximations serve to regularize the spectral edges.

Bloch wave method has been extended to other non-periodic media such as Hashin-Shtrikman structures~\cite{Ghosh2018}. A notion of Bloch-Taylor waves for aperiodic media has been introduced in~\cite{Gloria2017} using regularized correctors. A notion of approximate Bloch waves for quasiperiodic media using periodic lifting is also introduced by the authors in~\cite{Vivek2019ii}.

\subsection{Plan of the Paper}
In Section~\ref{section2}, we shall introduce some of the notation and definitions are required in the text. In Section~\ref{section3}, we will define periodic approximations for almost periodic functions. In Section~\ref{section4}, Bloch wave analysis of periodic approximations is performed. In Section~\ref{section5}, we prove the homogenization result by first taking the limit $\epsilon\to 0$, followed by the limit $R\to \infty$ in the Bloch transform of the periodic approximations. In Section~\ref{section6}, we prove that the homogenized coefficients of the periodic approximations converge to those of the almost periodic operator. In Section~\ref{section7}, we prove that the higher Bloch modes do not contribute to the homogenization process. In Section~\ref{roc}, we will establish rate of convergence for approximations of homogenized tensors corresponding to periodizations. Finally, in Section~\ref{numerics}, we conduct a numerical study of Dirichlet and Periodic approximations.

\section{Notations and Definitions}\label{section2}
\subsection{Periodicity} Let $Y=[-\pi,\pi)^d$ denote a parametrization for the $d$-dimensional torus $\mathbb{T}^d$. Measurable and bounded $Y$-periodic functions in $\mathbb{R}^d$ will be denoted by $L^\infty_\sharp(Y)$, which is another manifestation of $L^\infty(\mathbb{T}^d)$. The space of all $L^2_{\loc}(\mathbb{R}^d)$ functions that are $Y$-periodic are denoted by $L^2_\sharp(Y)$. Similarly, the space of all $H^1_{\loc}(\mathbb{R}^d)$ functions that are $Y$-periodic are denoted by $H^1_\sharp(Y)$.

\subsection{Almost Periodicity} 

For $K=\mathbb{R}$ or $\mathbb{C}$, let $\Trig(\mathbb{R}^d;K)$ denote the space of all $K$-valued trigonometric polynomials of the form $\displaystyle P(y)=\sum_{j=1}^{N}a_j e^{iy\cdot \eta_j}$. 

\begin{definition}
	A bounded continuous function $u:\mathbb{R}^d\to\mathbb{R}$ is said to be uniformly almost periodic if it is the uniform limit of a sequence of real trigonometric polynomials, i.e., there exists a sequence $P_n(y)\in \Trig(\mathbb{R}^d;\mathbb{R})$ such that $||u-P_n||_\infty\to 0$ as $n\to\infty$.
\end{definition}

Uniformly almost periodic functions are also known as Bohr almost periodic functions. The set of all Bohr almost periodic functions when equipped with the uniform norm is a Banach space denoted by $AP(\mathbb{R}^d)$. 

\begin{definition}
	The mean value of a function $u:\mathbb{R}^d\to\mathbb{R}$ in $L^1_{\loc}(\mathbb{R}^d)$ is the following limit
	\begin{equation}
	\mathcal{M}(u)=\limsup_{L\to\infty}\frac{1}{|LY|}\int_{LY}u(y)~dy,
	\end{equation}
	where $LY=[-L\pi,L\pi)^d$ and $|\cdot|$ denotes its Lebesgue measure.
\end{definition}

In fact, the limsup above is a limit for Bohr almost periodic functions~\cite{Besicovitch55}. One can obtain a larger class of functions than Bohr almost periodic functions by employing the notion of mean value.

\begin{definition}
	A function $u\in L^2_{\loc}(\mathbb{R}^d;\mathbb{C})$ is said to be Besicovitch almost periodic if there exists a sequence $P_n(y)\in \Trig(\mathbb{R}^d;\mathbb{C})$ such that $\mathcal{M}(|u-P_n|^2)\to 0$ as $n\to\infty$.
\end{definition}

On the set of all Besicovitch almost periodic functions, the quantity $\displaystyle(\mathcal{M}(|\cdot|^2))^{1/2}$ is a semi-norm. Given Besicovitch almost periodic functions $f$ and $g$, we shall identify them if $\mathcal{M}(|f-g|^2)=0$ to obtain a Hilbert space, with the inner product given by $\mathcal{M}(f\cdot\overline{g})$, which will be denoted by $B^2(\mathbb{R}^d)$. The superscript $2$ serves to remind us that one could very well define a Besicovitch analogue of $L^p$ spaces.

Let us recall some interesting properties of almost periodic functions. In direct analogy with periodic functions, one can define a formal Fourier series for almost periodic functions~\cite{Besicovitch55}. The trigonometric factors $e^{iy\cdot\eta}$ that appear in the Fourier series of an almost periodic function $u$ correspond to all $\eta\in\mathbb{R}^d$ for which $\mathcal{M}(ue^{-iy\cdot\eta})$ is non-zero. Note that for a given function $u$, the set of all such $\eta$ is countable. This set is called the set of frequencies of $u$ and the $\mathbb{Z}$-module generated by these frequencies is denoted as $Mod(u)$. Almost periodic functions $u$ with a finitely generated $Mod(u)$ are called quasiperiodic functions. It is interesting to note that $AP(\mathbb{R}^d)$ and $B^2(\mathbb{R}^d)$ are examples of non-separable Banach spaces. More information about these function spaces may be found in~\cite{Besicovitch55,Levitan82,Cord2009}. Further, a short but illuminating crash course on almost periodic functions may be found in~\cite{Shubin78}.

\subsection{Almost Periodic Differential Operators}
Consider the almost periodic second-order elliptic operator in divergence form given by
\begin{equation}
\mathcal{A}u:=-\dive(A\nabla u)=-\frac{\partial}{\partial y_k}\left(a_{kl}(y)\frac{\partial u}{\partial y_l}\right),\label{eq2:operator2}
\end{equation} where summation over repeated indices is assumed and the coefficients satisfy the following assumptions:
\begin{enumerate}[label={(A\arabic*)}] 
\item\label{A1} The coefficients $A=(a_{kl}(y))$ are measurable bounded real-valued almost periodic functions defined on $\mathbb{R}^d$. In other words, $a_{kl}\in AP(\mathbb{R}^d)$. 
\item\label{A2} The matrix $A=(a_{kl})$ is symmetric, i.e., $a_{kl}(y)=a_{lk}(y)$ $\forall\,y\in\mathbb{R}^d$. 
\item\label{A3} Further, the matrix $A$ is {\it coercive}, i.e., there exists an $\alpha>0$ such that  
\begin{equation}\label{coercivity}
\forall\, v\in\mathbb{R}^d\mbox{ and } a.e.\, y\in\mathbb{R}^d,\langle A(y)v, v\rangle\geq \alpha||v||^2.
\end{equation}
\end{enumerate}

Let $\Omega$ be an open set in $\mathbb{R}^d$. We are interested in the homogenization of the following equation posed in $H^1(\Omega)$
\begin{equation}
\mathcal{A}^\epsilon u^\epsilon:=-\frac{\partial}{\partial x_k}\left(a_{kl}^\epsilon\left({\epsilon}\right)\frac{\partial u^\epsilon}{\partial x_l}\right)=f,\label{eq3:oscillating}
\end{equation} where $f\in L^2(\Omega)$ and $a_{kl}^\epsilon\left({\epsilon}\right)\coloneqq a_{kl}\left(\frac{x}{\epsilon}\right)$. Suppose that $u^\epsilon$ converges weakly to a limit $u\in H^1(\Omega)$. We shall prove in the course of this paper that $u$ satisfies an equation of the form 

\begin{equation}
\mathcal{A}^* u:=-\frac{\partial}{\partial x_k}\left(a_{kl}^*\left(x\right)\frac{\partial u}{\partial x_l}\right)=f,\label{eq4:homogenized}
\end{equation} and we also identify the coefficients $a_{kl}^*$. The assumption of symmetry is not essential for the purposes of homogenization since it is possible to define a dominant Bloch mode~\cite{Sivaji2004} in the non-selfadjoint case.

Homogenization of almost periodic media was first carried out by Kozlov~\cite{Kozlov78} using quasiperiodic approximations. Subsequently, an abstract approach was given in~\cite{Oleinik1982,Jikov1994} which is described in Subsection~\ref{APhom}.

Some further notation that we make use of, is listed below:
\begin{itemize}
   \item We shall call a bounded continuous matrix-valued function $A$ almost periodic if each of its entries is an almost periodic function.
   \item The notation $\lesssim$ is shorthand for $\leq$ with a multiplicative constant which does not depend on $\epsilon$ and $R$ but may depend on the dimension $d$, $L^\infty$ bound of $A$, the coercivity constant $\alpha$, etc.
   \item The notation $\displaystyle\dashint_Gb(t)\,dt$ denotes the average $\displaystyle\frac{1}{|G|}\int_Gb(t)\,dt$ of a function $b$ over $G\subset\mathbb{R}^d$. Sometimes, the notation $\mathcal{M}_G(b)$ is also used.
   \item For $L>0$, let $Y_{L}$ denote the set $[-L\pi, L\pi)^d$. 
\end{itemize}

\section{Periodic Approximations of Almost Periodic Functions}\label{section3}
Equation~\eqref{eq3:oscillating} is not amenable to a Bloch wave analysis due to non-periodicity of the coefficients. Hence, we shall introduce some periodic approximations to the coefficients of the operator~\eqref{eq2:operator2}. These periodic approximations follow the simple principle of ``restrict and periodize". Given $f\in {AP}(\mathbb{R}^d)$, define 
\begin{align}\label{perapprox}f^R(y)=f(y) \mbox{ for }  y\in Y_R=[-R\pi ,R\pi )^d,\end{align} and extend to the whole of $\mathbb{R}^d$ by periodization, i.e, $f^R(y+2\pi  Rp)=f(y)$ for all $p\in\mathbb{Z}^d$. Hence, the periodic approximation so constructed belongs to $L^\infty_\sharp(Y_R)$.

The sequence $f^R$ may not converge in $L^\infty(\mathbb{R}^d)$. In fact, the functions which can be written as a uniform limit of periodic functions are called as limit-periodic functions~\cite{Levitan82} and they form a subclass of almost periodic functions. However, the sequence is convergent in $L^2_{\loc}(\mathbb{R}^d)$ as well as uniformly on compact subsets of $\mathbb{R}^d$. It is unclear if the sequence $f^R$ converges to $f$ in $B^2(\mathbb{R}^d)$. We remark here that convergence in $L^2_{\loc}(\mathbb{R}^d)$ does not imply convergence in $B^2(\mathbb{R}^d)$.

\begin{remark}
	Another periodic approximation to almost periodic functions is constructed in~\cite{Shubin78}. Roughly speaking, given a $f\in AP(\mathbb{R}^d)$, there is a sequence of numbers $(T_n)_{n\in\mathbb{N}}$ going to $\infty$ and a sequence of $T_nY$-periodic functions $P_n$ such that $||u-P_n||_{\infty,T_nY}\to 0$ as $n\to\infty$. The notation $||\cdot||_{\infty,T_nY}$ implies that the $L^\infty$ norm is taken over the cube $T_nY=\left[-T_n\pi,T_n\pi\right)^d$. The proof of this theorem involves approximation of irrationals by rationals by Dirichlet's Approximation Theorem. Clearly, either of these approximations may be used for our purposes. Note that the approximations in~\cite{Shubin78} have the advantage that they are smooth being trigonometric polynomials; however, the sequence $(T_n)_{n\in\mathbb{N}}$ cannot be chosen.
\end{remark}

\subsection{Periodic Approximations of Almost Periodic Operators}
\hspace{-0.2cm}For $R>0$, we denote by $A^R=(a_{kl}^R(y))_{k,l=1}^d$ the periodic approximation of $A=(a_{kl}(y))_{k,l=1}^d$ at level $R$, as explained in~\eqref{perapprox}, i.e., for $1\leq k,l\leq d$,
\begin{align}
	\begin{cases}
	a_{kl}^R(y)=a_{kl}(y)&\mbox{ for }y\in Y_R\\
	a_{kl}^R(y+2\pi Rp)=a_{kl}(y)&\mbox{ for }p\in\mathbb{Z}^d
	\end{cases}
\end{align}
The following operator will serve as a periodic approximation to $\mathcal{A}$.
\begin{equation*}
\mathcal{A}^Ru:=-\dive(A^R\nabla u)=-\frac{\partial}{\partial y_k}\left(a^R_{kl}(y)\frac{\partial u}{\partial y_l}\right).
\end{equation*} Such an approximation has been considered in~\cite{BourgeatPiatnitski2004}.

\section{Bloch wave Analysis for Periodic Approximations}\label{section4}
In this section, we shall perform a Bloch wave analysis for the periodic approximations of the operator in~\eqref{eq2:operator2}. In particular, we shall study, for each fixed $R>0$, the Bloch waves for the operators in $L^2(\mathbb{R}^d)$ given by
\begin{equation}
\mathcal{A}^Ru:=-\dive(A^R\nabla u)=-\frac{\partial}{\partial y_k}\left(a^R_{kl}(y)\frac{\partial u}{\partial y_l}\right).\label{Approx_n}
\end{equation} 

Let $Y_R^{'}\coloneqq \left[-\frac{1}{2R},\frac{1}{2R}\right)^d$ denote a basic cell for the dual lattice corresponding to $2\pi R\mathbb{Z}^d$. The operator $\mathcal{A}^R$ can be written as the direct integral $\displaystyle\int_{Y_R^{'}}^\bigoplus\mathcal{A}^R(\eta)~d\eta$, where \begin{align}\mathcal{A}^R(\eta)=e^{-i\eta\cdot y}\mathcal{A}^Re^{i\eta\cdot y}=-\left(\frac{\partial}{\partial y_k}+i\eta_k\right)a^R_{kl}(y)\left(\frac{\partial}{\partial y_l}+i\eta_l\right),\label{shiftedoperator}\end{align} is an unbounded operator in $L^2_\sharp(Y_R)$. As a consequence, the spectrum of the operator $\mathcal{A}^R$ is the union of spectra of $\mathcal{A}^R(\eta)$ as $\eta$ varies in $Y_R^{'}$~\cite[p.~284]{Reed1978}. It can be shown that the operators $A^R(\eta)$ have compact resolvent~\cite{Bensoussan2011}. Therefore, $\mathcal{A}^R(\eta)$ has a sequence of eigenvalues and eigenvectors
\begin{align}
\eta\mapsto(\lambda^R_m(\eta),\phi^R_m(y;\eta)), m=1,2,\ldots, \label{Blochevalues}
\end{align}
which are called Bloch eigenvalues and eigenvectors.

\begin{remark}\label{normalization}
	We shall choose $||\phi_1^R(\cdot;\eta)||_{L^2_\sharp(Y_R)}=R^{d/2}$ and $\phi^R_1(y;0)=\frac{1}{(2\pi)^{d/2}}$ $\forall\, R>0$.
\end{remark}

\subsection{Bloch Decomposition of $L^2(\mathbb{R}^d)$}

In this section, we shall state the theorem on decomposition of functions in $L^2(\mathbb{R}^d)$ using Bloch waves. We shall not go through the details of the proof, which may be found in~\cite{Bensoussan2011},~\cite{Sivaji2004} and~\cite{SivajiGanesh2005}.

Consider the unbounded operator defined in $L^2(\mathbb{R}^d)$ \begin{equation}
\mathcal{A}^{R,\epsilon}u:=-\dive(A^{R,\epsilon}\nabla u)=-\frac{\partial}{\partial x_k}\left(a^{R,\epsilon}_{kl}\left(x\right)\frac{\partial u}{\partial x_l}\right),\label{Approx_epsilon}
\end{equation} where $\displaystyle a^{R,\epsilon}_{kl}(x)\coloneqq a^{R}_{kl}\left(\frac{x}{\epsilon}\right)$.

By homothecy, the Bloch eigenvalues and Bloch eigenvectors for the operator~\eqref{Approx_epsilon} are
\begin{align}
\lambda^{R,\epsilon}_m(\xi)=\epsilon^{-2}\lambda^R_m(\epsilon\xi),\,\phi^{R,\epsilon}_m(x;\xi)=\phi^R_m\left(\frac{x}{\epsilon};\epsilon\xi\right),
\end{align}
where $\lambda^R_m(\eta)$ and $\phi^R_m(\eta)$ are defined in~\eqref{Blochevalues}.

\begin{theorem}\label{BlochDecomposition} Let $R>0$. Let $g\in L^2(\mathbb{R}^d)$. Define the $m^{th}$ Bloch coefficient of $g$ as 
	\begin{align}\label{BlochTransform1epsilon}
	\mathcal{B}^{R,\epsilon}_mg(\xi)\coloneqq\int_{\mathbb{R}^d}g(x)e^{-ix\cdot\xi}\overline{\phi_m^{R,\epsilon}(x;\xi)}~dx,~m\in\mathbb{N},~\xi\in \epsilon^{-1}Y^{'}_R.
	\end{align}
	\begin{enumerate}
		\item  The following inverse formula holds
		\begin{align}\label{Blochinverse}
		g(y)=\int_{\epsilon^{-1}Y_R^{'}}\sum_{m=1}^{\infty}\mathcal{B}^{R,\epsilon}_mg(\xi)\phi_m^{R,\epsilon}(x;\xi)e^{ix\cdot\xi}\,d\xi.
		\end{align}
		\item{\bf Parseval's identity} \begin{align}\label{unitary1epsilon}
		||g||^2_{L^2(\mathbb{R}^d)}=\sum_{m=1}^{\infty}\int_{\epsilon^{-1}Y_R^{'}}|\mathcal{B}^{R,\epsilon}_mg(\xi)|^2\,d\xi.
		\end{align}
		\item{\bf Plancherel formula} For $f,g\in L^2(\mathbb{R}^d)$, we have\begin{align}\label{Plancherel}
		\int_{\mathbb{R}^d}f(y)\overline{g(y)}\,dy=\sum_{m=1}^{\infty}\int_{\epsilon^{-1}Y_R^{'}}\mathcal{B}^{R,\epsilon}_mf(\xi)\overline{\mathcal{B}^{R,\epsilon}_mg(\xi)}\,d\xi.
		\end{align}
		\item{\bf Bloch Decomposition in $H^{-1}(\mathbb{R}^d)$} For an element $F=u_0(x)+\sum_{j=1}^N\frac{\partial u_j(x)}{\partial x_j}$ of $H^{-1}(\mathbb{R}^d)$, the following limit exists in $L^2(\epsilon^{-1}Y_R^{'})$:
		\begin{align}\label{BlochTransform2epsilon}
		\mathcal{B}^{R,\epsilon}_mF(\xi)=\int_{\mathbb{R}^d}e^{-ix\cdot\xi}\left\{u_0(x)\overline{\phi^{R,\epsilon}_m(x;\xi)}+i\sum_{j=1}^N\xi_ju_j(x)\overline{\phi^{R,\epsilon}_m(x;\xi)}\right\}\,dx\nonumber\\-\int_{\mathbb{R}^d}e^{-ix\cdot\xi}\sum_{j=1}^Nu_j(x)\frac{\partial\overline{\phi^{R,\epsilon}_m}}{\partial x_j}(x;\xi)\,dx.
		\end{align}
		\item[] The definition above is independent of the particular representative of $F$. 
		\item Finally, for $g\in D(\mathcal{A}^{R,\epsilon})$, \begin{align}
		\label{diagonalizationepsilon}
		\mathcal{B}^{R,\epsilon}_m(\mathcal{A}^{R,\epsilon}g)(\xi)=\lambda^{R,\epsilon}_m(\xi)\mathcal{B}^{R,\epsilon}_mg(\xi).\end{align}
	\end{enumerate}
\end{theorem}

\subsection{Bloch Transform converges to Fourier Transform}

The following lemma says that as $\epsilon\to 0$, the first Bloch coefficient of a function converges to its Fourier transform, which is defined as $\displaystyle\hat{u}(\xi)=\int_{\mathbb{R}^d}u(y)e^{-ix\cdot\xi}~dy$. This is a consequence of the Lipschitz continuity of $\phi^R_1(y;\eta)$ in $\eta$ close to $\eta=0$, and the choice of normalization of the first Bloch eigenfunctions (See Remark~\ref{normalization}). For a proof, see~\cite{Conca1997}.

\begin{lemma}\label{Convergence_Bloch}
    Let $R>0$. Let $g,g^\epsilon\in L^2(\mathbb{R}^d)$ be such that the support of $g^\epsilon$ is contained in a fixed compact subset $K\subset\mathbb{R}^d$, independent of $\epsilon$. If $g^\epsilon$ converges weakly to $g$ in $L^2(\mathbb{R}^d)$, then we have
	$\chi_{\epsilon^{-1}U_R}\mathcal{B}^{R,\epsilon}_1g^\epsilon(\xi)\rightharpoonup\hat{g}(\xi)$ in $L^2(\mathbb{R}^d_\xi)$-weak.
\end{lemma}

\subsection{Regularity Properties of Bloch eigenvalues and eigenvectors}
In physical applications, the regularity properties of Bloch eigenvalues and eigenvectors with respect to the dual parameter $\eta\in Y_R^{'}$ plays an important role, for example, see:~\cite{AllaireVanni2004},~\cite{AllairePiatnitski2005},~\cite{AllaireRauch2011}. It is a simple consequence of the Courant-Fischer minmax principle that Bloch eigenvalues are Lipschitz continuous in the dual parameter~\cite{Conca1997}. However, such limited regularity is usually not sufficient for our purposes. We require the following theorem about the behavior of the first Bloch eigenvalue and eigenvector in a neighborhood of $0\in Y_R^{'}$.

\begin{theorem}\label{analyticity}
	There is a neighborhood $U_R\coloneqq\{\eta\in Y_R^{'}:|\eta|<\delta_R\}$, where $\delta_R$ is a positive real number, such that the first Bloch eigenvalue $\lambda_1^R(\eta)$ is analytic for $\eta\in U_R$ and the first Bloch eigenvector $\phi^R_1(\eta)\in H^1_\sharp(Y_R)$ may be chosen to be analytic for $\eta\in U_R$.
\end{theorem} 

A proof of Theorem~\ref{analyticity} that uses the notion of infinite-dimensional determinants can be found in~\cite{Conca1997}. Another proof that uses the Kato-Rellich Theorem~\cite{Reed1978},~\cite{Kato1995} may be found in~\cite{Sivaji2004}.

\begin{remark}
	The radius $\delta_R$ of the neighborhood $U_R$ depends on the gap between the first and second Bloch eigenvalues of the operator $\mathcal{A}^R$. The limit operator of $\mathcal{A}^R$ is the almost periodic operator $\mathcal{A}$ which often has a Cantor-like spectrum~\cite{Damanik2019}. Hence, we expect the spectral gap to vanish in the limit $R\to\infty$. Therefore, the neighborhood $U_R$ is expected to shrink to $0$ in the limit $R\to\infty$. 
\end{remark}

\subsection{Derivatives of the first Bloch eigenvalue and eigenfunction}

In the theory of periodic homogenization~\cite{Bensoussan2011}, homogenized coefficients are given in terms of solutions of the cell problem which is an equation posed on the basic periodic cell. For the $R^{th}$ periodic approximation~\eqref{Approx_n}, we recall the cell problem and the homogenized coefficients below.

The homogenized coefficients for the $R^{th}$ periodic approximation are given by:
\begin{align}\label{homoR}
a_{kl}^{R,*}=\frac{1}{|Y_R|}\int_{Y_R}a_{kl}^R(y)\,dy+\frac{1}{|Y_R|}\int_{Y_R}a_{kp}^R(y)\frac{\partial w^{R,l}}{\partial y_p}\,dy,
\end{align}

where $w^{R,p}\in H^1_\sharp(Y_R)$ satisfy the following cell problems for $1\leq p\leq d$:

\begin{align}\label{CellP}
\mathcal{A}^Rw^{R,p}=-\frac{\partial}{\partial y_k}\left(a^R_{kl}\left(y\right)\frac{\partial w^{R,p}}{\partial y_l}\right)=\frac{\partial a_{lp}^R}{\partial y_l}(y)\,\mbox{ in }\,Y_R.
\end{align}

The functions $w^{R,p}$ are called {\it correctors} and $w^R$ is called {\it corrector field}. We recall that $\lambda^R_1(\eta)$ and $\phi_1^R(\eta)$ are analytic in $U_R\subset Y_R^{'}$. The proof of the following theorem is standard and may be found in~\cite{Conca1997} or~\cite{Sivaji2004}.

\begin{theorem}\label{Hessian}
	The first Bloch eigenvalue and eigenfunction of the $R^{th}$ periodic approximation $\mathcal{A}^R$ satisfy:
	\begin{enumerate}
		\item $\lambda^R_1(0)=0$.
		\item The eigenvalue $\lambda^R_1(\eta)$ has a critical point at $\eta=0$, i.e., \begin{align}\frac{\partial \lambda^R_1}{\partial \eta_s}(0)=0, \forall s=1,2,\ldots,d.\end{align}
		\item For $s=1,2,\ldots,d,$ the derivative of the eigenvector $(\partial \phi_1^R/\partial\eta_s)(0)$ satisfies:
		
		$(\partial \phi_1^R/\partial\eta_s)(y;0)-i\phi^R_1(y;0)w^{R,s}(y)$ is a constant in $y$.
		\item The Hessian of the first Bloch eigenvalue at $\eta=0$ is twice the homogenized matrix $a_{kl}^{R,*}$:
		\begin{align}\label{identificationofapproximatehomo}
		\frac{1}{2}\frac{\partial^2\lambda^R_1}{\partial\eta_k\partial\eta_l}(0)=a_{kl}^{R,*}.
		\end{align}
	\end{enumerate}
\end{theorem}

\subsection{Boundedness of Corrector Field}\label{boundednessofcellfunctions}
We will show that the sequence $(\nabla w^{R,p})_{R>0}$ is bounded in $B^2(\mathbb{R}^d)$, independent of $R$. We know that for each $R>0$ and $1\leq p\leq d$, $\nabla w^{R,p}\in (L^2_\sharp(Y_R))^d\subset (B^2(\mathbb{R}^d))^d$ satisfies
\begin{align}\label{cellproblemforperiodicapproximation}
\mathcal{M}\left(A^R\nabla w^{R,p}\nabla w^{R,p}\right)&=-\sum_{l=1}^d\mathcal{M}\left(a^R_{lp}\frac{\partial w^{R,p}}{\partial y_l}\right)
\end{align}
Using the coercivity and boundedness of the matrix $A$, we obtain:
\begin{align*}
\alpha||\nabla w^{R,p}||^2_{(L^2_\sharp(Y_R))^d}&\leq C ||\nabla w^{R,p}||_{(L^2_\sharp(Y_R))^d}
\end{align*} From the last equation, we obtain the norm-boundedness of $(\nabla w^{R,p})$ in $(L^2_\sharp(Y_R))^d$ and hence in $(B^2(\mathbb{R}^d))^d$. 

\subsection{Boundedness of homogenized tensors}\label{boundednessoftensors}

Due to the boundedness of derivatives of the correctors proved in Subsection~\ref{boundednessofcellfunctions}, the sequence of numbers $a_{kl}^{R,*}$, defined in~\eqref{homoR}, is bounded independently of $R$. Further, it follows from the identification~\eqref{identificationofapproximatehomo} that the sequence of numbers $\frac{1}{2}\frac{\partial^2\lambda^R_1}{\partial\eta_k\partial\eta_l}(0)$ is bounded. Hence, there is a subsequence, still labeled by $R$, for which the sequence $\frac{1}{2}\frac{\partial^2\lambda^R_1}{\partial\eta_k\partial\eta_l}(0)$ converges. We shall call this limit as $a_{kl}^*$, i.e.,\begin{align}\lim_{R\to\infty}\frac{\partial^2\lambda^R_1}{\partial\eta_k\partial\eta_l}(0)=2a_{kl}^*.\label{somecoefficients}\end{align}

\section{Homogenization Result}\label{section5}
In this section, we shall state the homogenization result for almost periodic media and prove it using the Bloch wave method. It will be seen in a further section that the coefficients $a_{kl}^*$, defined in~\eqref{somecoefficients}, coincide with the homogenized coefficients for almost periodic media~\cite{Oleinik1982}. In this section, we shall assume summation over repeated indices for ease of notation.

\begin{theorem}\label{homog}
	Let $\Omega$ be an arbitrary domain in $\mathbb{R}^d$ and $f\in L^2(\Omega)$. Let $u^\epsilon\in H^1(\Omega)$ be such that $u^\epsilon$ converges weakly to $u^*$ in $H^1(\Omega)$, and
	\begin{align}\label{equation}
	\mathcal{A}^\epsilon u^\epsilon=f\,\mbox{in}\,~\Omega.
	\end{align}
	Then
	\begin{enumerate}
		\item For all $k=1,2,\ldots,d$, we have the following convergence of fluxes:
		\begin{align}
		a^\epsilon_{kl}(x)\frac{\partial u^\epsilon}{\partial x_l}(x)\rightharpoonup a_{kl}^*\frac{\partial u^*}{\partial x_l}(x) \mbox{ in } L^2(\Omega)\mbox{-weak}.
		\end{align}
		\item The limit $u^*$ satisfies the homogenized equation:
		\begin{align}\label{homoperator}
		\mathcal{A}^{hom}u^*=-\frac{\partial}{\partial x_k}\left(a^*_{kl}\frac{\partial u^*}{\partial x_l}\right)=f\,\mbox{ in }\,\Omega,
		\end{align}
	\end{enumerate} where $(a^*_{kl})_{1\leq k,l\leq d}$ are given in~\eqref{somecoefficients}.
\end{theorem}

The proof of Theorem~\ref{homog} is divided into the following steps. We begin by localizing the equation~\eqref{equation} which is posed on $\Omega$, so that it is posed on $\mathbb{R}^d$. We take the first Bloch transform $\mathcal{B}^{R,\epsilon}_1$ of this equation and pass to the limit $\epsilon\to 0$, followed by the limit $R\to\infty$. The proof relies on the analyticity of the first Bloch eigenvalue and eigenfunction in a neighborhood of $0\in Y_R^{'}$. The limiting equation is an equation in Fourier space. The homogenized equation is obtained by taking the inverse Fourier transform.

\subsection{Localization}

Let $\psi_0$ be a fixed smooth function supported in a compact set $K\subset\mathbb{R}^d$. Since $u^\epsilon$ satisfies $\mathcal{A}^\epsilon u^\epsilon=f$, $\psi_0 u^\epsilon$ satisfies
\begin{align}\label{local}
\mathcal{A}^{R,\epsilon}(\psi_0 u^\epsilon)(x)=\psi_0f(x)+g^\epsilon(x)+h^{R,\epsilon}(x)+l^{R,\epsilon}(x)\,\mbox{ in }\,\mathbb{R}^d,
\end{align} where
\begin{align}
g^\epsilon(x)&\coloneqq-\frac{\partial \psi_0}{\partial x_k}(x)a^\epsilon_{kl}(x)\frac{\partial u^\epsilon}{\partial x_l}(x),\label{geps}\\
h^{R,\epsilon}(x)&\coloneqq-\frac{\partial}{\partial x_k}\left(\frac{\partial \psi_0}{\partial x_l}(x)a^{R,\epsilon}_{kl}(x)u^\epsilon(x)\right),\label{heps}\\
l^{R,\epsilon}(x)&\coloneqq-\frac{\partial}{\partial x_k}\left(\psi_0(x)\left(a^{R,\epsilon}_{kl}(x)-a^\epsilon_{kl}(x)\right)\frac{\partial u^\epsilon}{\partial x_l}(x)\right)\label{leps}.
\end{align}

While the sequence $g^\epsilon$ is bounded in $L^2(\mathbb{R}^d)$, the sequences $h^{R,\epsilon}$ and $l^{R,\epsilon}$ are bounded in $H^{-1}(\mathbb{R}^d)$. Taking the first Bloch transform of both sides of the equation~\eqref{local}, we obtain for $\xi\in\epsilon^{-1}U_R$ a.e.
\begin{align}\label{Bloch}
\lambda^{R,\epsilon}_1(\xi)\mathcal{B}^{R,\epsilon}_1(\psi_0 u^\epsilon)(\xi)=\mathcal{B}^{R,\epsilon}_1(\psi_0 f)(\xi)+\mathcal{B}^{R,\epsilon}_1g^\epsilon(\xi)+\mathcal{B}^{R,\epsilon}_1h^{R,\epsilon}(\xi)+\mathcal{B}^{R,\epsilon}_1l^{R,\epsilon}(\xi)
\end{align}

We shall now pass to the limit $\epsilon\to 0$, followed by the limit $R\to\infty$ in the equation~\eqref{Bloch}.

\subsection{Limit $\epsilon\to 0$}
\subsubsection{Limit of $\lambda^{R,\epsilon}_1(\xi)\mathcal{B}^{R,\epsilon}_1(\psi_0 u^\epsilon)$}
We substitute the power series expansion of the first Bloch eigenvalue about $\eta=0$ in $\lambda^{R,\epsilon}_1(\xi)\mathcal{B}^{R,\epsilon}_1(\psi_0 u^\epsilon)$ and then pass to the limit $\epsilon\to 0$ in $L^2_{\loc}(\mathbb{R}^d_\xi)$-weak by applying Lemma~\ref{Convergence_Bloch} to obtain:
\begin{align}\label{convergence_ueps}
\frac{1}{2}\frac{\partial^2\lambda^R_1}{\partial\eta_s\partial\eta_t}(0)\xi_s\xi_t\widehat{\psi_o u^*}(\xi).
\end{align}
\subsubsection{Limit of $\mathcal{B}^{R,\epsilon}_1(\psi_0 f)$}
A simple application of Lemma~\ref{Convergence_Bloch} yields the convergence of $\mathcal{B}^{R,\epsilon}_1(\psi_0 f)$ to $(\psi_0 f)^{\bf\widehat{}}$ in $L^2_{\loc}(\mathbb{R}^d_\xi)$-weak.
\subsubsection{Limit of $\mathcal{B}^{R,\epsilon}_1g^\epsilon$}
The sequence $g^\epsilon$ as defined in~\eqref{geps} is bounded in $L^2(\mathbb{R}^d)$ and hence has a weakly convergent subsequence with limit $g^*\in L^2(\mathbb{R}^d)$. This sequence is supported in a fixed set $K$. Also, note that the sequence $\displaystyle\sigma_k^\epsilon(x)\coloneqq a^\epsilon_{kl}(x)\frac{\partial u^\epsilon}{\partial x_l}(x)$ is bounded in $L^2(\Omega)$, hence has a weakly convergent subsequence whose limit is denoted by $\sigma^*_k$ for $k=1,2,\ldots,d$. Extend $\sigma^*_k$ by zero outside $\Omega$ and continue to denote the extension by $\sigma^*_k$. Thus, $g^*$ is given by $-\frac{\partial \psi_0}{\partial x_k}\sigma^*_k$. Therefore, by Lemma~\ref{Convergence_Bloch}, we obtain the following convergence in $L^2_{\loc}(\mathbb{R}^d_\xi)$-weak:
\begin{align}\label{convergence_geps}
\chi_{\epsilon^{-1}U_R}(\xi)\mathcal{B}^{R,\epsilon}_1g^\epsilon(\xi)\rightharpoonup-\left(\frac{\partial \psi_0}{\partial x_k}(x)\sigma^*_k(x)\right)^{\bf\widehat{}}(\xi).
\end{align}

Notice that the limit is independent of $R$.

\subsubsection{Limit of $\mathcal{B}_1^{R,\epsilon}h^{R,\epsilon}$}

We have the following weak convergence for $\mathcal{B}_1^{R,\epsilon}h^{R,\epsilon}$ in $L^2_{\loc}(\mathbb{R}^d_{\xi})$.
\begin{align}\label{convergence_heps}
\lim_{\epsilon\to 0}\,\chi_{\epsilon^{-1}U_R}(\xi)\mathcal{B}^{R,\epsilon}_1h^{R,\epsilon}(\xi)=-i\xi_ka^{R,*}_{kl}\left(\frac{\partial \psi_0}{\partial x_l}(x)u^*(x)\right)^{\bf\widehat{}}(\xi)
\end{align}

We shall prove this in the following steps.

\paragraph{Step 1} By the definition of the Bloch transform~\eqref{BlochTransform2epsilon} for elements of $H^{-1}(\mathbb{R}^d)$, we have
\begin{align}\label{heps1}
\mathcal{B}^{R,\epsilon}_1h^{R,\epsilon}(\xi)=-i\xi_k\int_{\mathbb{R}^d}e^{-ix\cdot\xi}\frac{\partial \psi_0}{\partial x_l}(x)a^{R,\epsilon}_{kl}(x)u^\epsilon(x)\overline{\phi^R_1\left(\frac{x}{\epsilon};\epsilon\xi\right)}\,dx\nonumber\\+\int_{\mathbb{R}^d}e^{-ix\cdot\xi}\frac{\partial \psi_0}{\partial x_l}(x)a^{R,\epsilon}_{kl}(x)u^\epsilon(x)\frac{\partial\overline{\phi^{R}_1}}{\partial x_k}\left(\frac{x}{\epsilon};\epsilon\xi\right)\,dx.
\end{align}

\paragraph{Step 2} The first term on RHS of~\eqref{heps1} is the Bloch transform of $-i\xi_k\frac{\partial \psi_0}{\partial x_l}(x)a^{R,\epsilon}_{kl}(x)u^\epsilon(x)$ which converges weakly to $-i\xi_k\mathcal{M}(a^R_{kl})\left(\frac{\partial \psi_0}{\partial x_l}(x)u^*(x)\right)$.

\paragraph{Step 3} Now, we analyze the second term on RHS of~\eqref{heps1}. In order to do this, we use the analyticity of first Bloch eigenfunction with respect to the dual parameter $\eta$ near $0$. We have the following power series expansion in $H^1_\sharp(Y_R)$ for $\phi_1^R(\eta)$ about $\eta=0$:
\begin{align}
\phi_1^R(y;\eta)=\phi_1^R(y;0)+\eta_s\frac{\partial \phi^R_1}{\partial \eta_s}(y;0)+\gamma^R(y;\eta).
\end{align}
We know that $\gamma^R(y;0)=0$ and $(\partial \gamma^R/\partial \eta_s)(y;0)=0$, therefore, $\gamma^R(\cdot;\eta)=O(|\eta|^2)$ in $L^\infty(U_R;H^1_\sharp(Y_R))$. We also have $(\partial\gamma^R/\partial y_k)(\cdot;\eta)=O(|\eta|^2)$ in $L^\infty(U_R;L^2_\sharp(Y_R))$.
Now,  
\begin{align}
\phi_1^{R,\epsilon}(x;\xi)=\phi_1^R\left(\frac{x}{\epsilon};\epsilon\xi\right)=\phi_1^R\left(\frac{x}{\epsilon};0\right)+\epsilon\xi_s\frac{\partial \phi^R_1}{\partial \eta_s}\left(\frac{x}{\epsilon};0\right)+\gamma^R\left(\frac{x}{\epsilon};\epsilon\xi\right).
\end{align}
Differentiating the last equation with respect to $x_k$, we obtain
\begin{align}\label{derivativeofphi}
\frac{\partial}{\partial x_k}\phi_1^R\left(\frac{x}{\epsilon};\epsilon\xi\right)=\xi_s\frac{\partial}{\partial y_k}\frac{\partial \phi^R_1}{\partial \eta_s}\left(\frac{x}{\epsilon};0\right)+\epsilon^{-1}\frac{\partial \gamma^R}{\partial y_k}\left(\frac{x}{\epsilon};\epsilon\xi\right).
\end{align}
For $\xi$ belonging to the set $\{\xi:\epsilon\xi\in U_R\mbox{ and }|\xi|\leq M\}$, we have
\begin{align}\label{gammaepsilon2}
\frac{\partial \gamma^R}{\partial y_k}(\cdot;\epsilon\xi)=O(|\epsilon\xi|^2)=\epsilon^2O(|\xi|^2)\leq CM^2\epsilon^2.
\end{align}
As a consequence,
\begin{align}
\epsilon^{-2}\frac{\partial \gamma^R}{\partial y_k}(x/\epsilon;\epsilon\xi)\in L^\infty_{\loc}(\mathbb{R}^d_\xi;L^2_\sharp(\epsilon Y_R)).
\end{align}
The second term on the RHS of~\eqref{heps1} is given by
\begin{align}\label{secondterm}
\chi_{\epsilon^{-1}U_R}(\xi)\int_K e^{-ix\cdot\xi}\frac{\partial\psi_0}{\partial x_l}(x)a^R_{kl}\left(\frac{x}{\epsilon}\right)u^\epsilon(x)\frac{\partial}{\partial x_k}\left(\overline{\phi^R_1}\left(\frac{x}{\epsilon};\epsilon\xi\right)\right)\,dx.
\end{align}
Substituting~\eqref{derivativeofphi} in~\eqref{secondterm}, we obtain
\begin{align}
\chi_{\epsilon^{-1}U_R}(\xi)\int_K e^{-ix\cdot\xi}\frac{\partial\psi_0}{\partial x_l}(x)a^R_{kl}\left(\frac{x}{\epsilon}\right)u^\epsilon(x)\left[\xi_s\frac{\partial}{\partial y_k}\frac{\partial \phi^R_1}{\partial \eta_s}\left(\frac{x}{\epsilon};0\right)+\epsilon^{-1}\frac{\partial \gamma^R}{\partial y_k}\left(\frac{x}{\epsilon};\epsilon\xi\right)\right]\,dx.
\end{align}
In the last expression, the term involving $\gamma^R$ goes to zero as $\epsilon\to 0$ in view of~\eqref{gammaepsilon2}, whereas the other term has the following limit due to the strong convergence of $u^\epsilon$ and weak convergence of $\left(a_{kl}^R(x/\epsilon)\frac{\partial}{\partial x_k}\left(\frac{\partial\phi^R_1}{\partial\eta_s}(x/\epsilon;0)\right)\right)$:
\begin{align}\label{secondterm2}
\mathcal{M}\left(a_{kl}^R(y)\frac{\partial}{\partial y_k}\left(\frac{\partial\phi^R_1}{\partial\eta_s}(y;0)\right)\right)\xi_s\int_{\mathbb{R}^d}e^{-ix\cdot\xi}\frac{\partial\psi_0}{\partial x_l}(x)u^*(x)\,dx.
\end{align}
\paragraph{Step 4} By Theorem~\ref{Hessian} and Remark~\ref{normalization}, it follows that 
\begin{align}\label{equivalence2}
\mathcal{M}\left(a_{kl}^R(y)\frac{\partial}{\partial y_k}\left(\frac{\partial\phi^R_1}{\partial\eta_s}(y;0)\right)\right)=-i(2\pi)^{-d/2}\mathcal{M}\left(a_{kl}^R(y)\frac{\partial w^{R,s}}{\partial y_k}(y)\right).
\end{align}
Therefore, we have the following convergence in $L^2_{\loc}(\mathbb{R}^d_\xi)$-weak:
\begin{align}
\chi_{\epsilon^{-1}U_R}(\xi)\mathcal{B}^{R,\epsilon}_1h^{R,\epsilon}(\xi)&\rightharpoonup-i\xi_s\left\{\mathcal{M}(a_{kl}^R)+\mathcal{M}\left(a_{kl}^R(y)\frac{\partial w^{R,s}}{\partial y_k}(y)\right)\right\}\left(\frac{\partial \psi_0}{\partial x_l}(x)u^*(x)\right)^{\bf\widehat{}}(\xi)\nonumber\\
&=-i\xi_s a_{kl}^{R,*} \left(\frac{\partial \psi_0}{\partial x_l}(x)u^*(x)\right)^{\bf\widehat{}}(\xi)
\end{align}

\subsubsection{Limit of $\mathcal{B}_1^{R,\epsilon}l^{R,\epsilon}$} Let $v^{R,\epsilon}_k\coloneqq \psi_0(x)(a^{R,\epsilon}_{kl}(x)-a^\epsilon_{kl}(x))\frac{\partial u^\epsilon}{\partial x_l}(x)$, then by the definition of the Bloch transform~\eqref{BlochTransform2epsilon} for elements of $H^{-1}(\mathbb{R}^d)$, we have
\begin{align}\label{leps2}
\mathcal{B}^{R,\epsilon}_1l^{R,\epsilon}(\xi)=-i\xi_k\int_{\mathbb{R}^d}e^{-ix\cdot\xi}v^{R,\epsilon}_k(x)\overline{\phi^R_1\left(\frac{x}{\epsilon};\epsilon\xi\right)}\,dx\nonumber\\+\int_{\mathbb{R}^d}e^{-ix\cdot\xi}v^{R,\epsilon}_k(x)\frac{\partial\overline{\phi^{R}_1}}{\partial x_k}\left(\frac{x}{\epsilon};\epsilon\xi\right)\,dx.
\end{align}

The sequence $v^{R,\epsilon}_k$ is bounded in $L^2(\mathbb{R}^d)$, hence converges weakly to a limit $v^R_k\in L^2(\mathbb{R}^d)$. The first term on the RHS of~\eqref{leps2} is the Bloch transform of $-i\xi_kv^{R,\epsilon}_k$, hence by Lemma~\ref{Convergence_Bloch}, it converges to $-i\xi_k(v^R_k(x))^{\widehat{}}(\xi)$.

Using equation~\eqref{derivativeofphi}, the second term on RHS of~\eqref{leps2} can be written as 
\begin{align}
\int_{\mathbb{R}^d}e^{-ix\cdot\xi}v^{R,\epsilon}_k(x)\left[\xi_s\frac{\partial}{\partial x_k}\frac{\partial \phi^R_1}{\partial \eta_s}\left(\frac{x}{\epsilon};0\right)+\epsilon^{-1}\frac{\partial \gamma^R}{\partial y_k}\left(\frac{x}{\epsilon};\epsilon\xi\right)\right]\,dx.
\end{align}

The second term in the above expression goes to $0$ in view of~\eqref{gammaepsilon2}. The sequence $z^{R,\epsilon}_s(x)\coloneqq v_k^{R,\epsilon}(x)\frac{\partial}{\partial x_k}\frac{\partial \phi^R_1}{\partial \eta_s}\left(\frac{x}{\epsilon};0\right)$ is bounded in $L^2(\mathbb{R}^d)$. Therefore, it has a weakly convergent subsequence whose limit we shall call $z^R_s$. The second term on RHS of~\eqref{leps2} converges to the Fourier transform of $\xi_sz^R_s$. 
\begin{align}\label{convergence_leps}
\lim_{\epsilon\to 0}\mathcal{B}_1^{R,\epsilon}l^{R,\epsilon}\to-i\xi_k(v^R_k(x))^{\widehat{}}(\xi)+\xi_s(z^R_s)^{\widehat{}}.
\end{align}

Finally, passing to the limit in~\eqref{Bloch} as $\epsilon\to 0$ by applying equations~\eqref{convergence_ueps},~\eqref{convergence_geps},~\eqref{convergence_heps} and~\eqref{convergence_leps} we get:
\begin{align}\label{Bloch2}
\frac{1}{2}\frac{\partial^2\lambda^R_1}{\partial\eta_s\partial\eta_t}(0)\xi_s\xi_t\widehat{\psi_o u^*}(\xi)=&(\psi_0 f)^{\widehat{}}(\xi)-\left(\frac{\partial \psi_0}{\partial x_k}(x)\sigma^*_k(x)\right)^{\bf\widehat{}}(\xi)-i\xi_s a_{kl}^{R,*} \left(\frac{\partial \psi_0}{\partial x_l}(x)u^*(x)\right)^{\bf\widehat{}}(\xi)\nonumber\\
&-i\xi_k(v^R_k(x))^{\widehat{}}(\xi)+\xi_s(z^R_s)^{\widehat{}}.
\end{align}

\subsection{Limit $R\to \infty$}
 In the equation~\eqref{Bloch2} above, we pass to the limit $R\to\infty$ as follows:
 
 Firstly, observe that
 $$||v_k^R||_{L^2(\mathbb{R}^d)}\leq\liminf_{\epsilon\to 0}||v_k^{R,\epsilon}||_{L^2(\mathbb{R}^d)}\leq C\max_{l}||a^{R,\epsilon}_{kl}-a^\epsilon_{kl}||_{L^\infty(K)}=C\max_{l}||a^{R}_{kl}-a_{kl}||_{L^\infty(\epsilon K)},$$ due to the weak lower semicontinuity of norm. Hence, $v_k^R\to 0$ as $R\to\infty$. 
 
 Secondly, $$||z_s^R||_{L^2(\mathbb{R}^d)}\leq\liminf_{\epsilon\to 0}||z_s^{R,\epsilon}||_{L^2(\mathbb{R}^d)}\leq C\max_{k,l}||a^{R,\epsilon}_{kl}-a^\epsilon_{kl}||_{L^\infty(K)}=C\max_{k,l}||a^{R}_{kl}-a_{kl}||_{L^\infty(\epsilon K)},$$ due to the weak lower semicontinuity of norm. Hence, $z_s^R\to 0$ as $R\to\infty$.
 
 As a consequence we obtain the following limit equation in the Fourier space:
\begin{align}\label{Bloch3}
a_{kl}^*\xi_k\xi_l\widehat{\psi_o u^*}(\xi)=\widehat{\psi_0 f}-\left(\frac{\partial \psi_0}{\partial x_k}(x)\sigma^*_k(x)\right)^{\bf\widehat{}}(\xi)-i\xi_ka^*_{kl}\left(\frac{\partial \psi_0}{\partial x_l}(x)u^*(x)\right)^{\bf\widehat{}}(\xi).
\end{align}

\subsection{Proof of the homogenization result}
Taking the inverse Fourier transform in the equation~\eqref{Bloch3} above, we obtain the following:
\begin{align}\label{eq1}
(\mathcal{A}^{hom}(\psi_0 u^*)(x))=\psi_0 f-\frac{\partial \psi_0}{\partial x_k}(x)\sigma^*_k(x)-a^*_{kl}\frac{\partial}{\partial x_k}\left(\frac{\partial \psi_0}{\partial x_l}(x)u^*(x)\right),
\end{align}
where the operator $\mathcal{A}^{hom}$ is defined in~\eqref{homoperator}. At the same time, calculating using Leibniz rule, we have:
\begin{align}\label{eq2}
(\mathcal{A}^{hom}(\psi_0 u^*)(x))=(\psi_0(x)\mathcal{A}^{hom}u^*(x))-a^*_{kl}\frac{\partial}{\partial x_k}\left(\frac{\partial \psi_0}{\partial x_l}(x)u^*(x)\right)-a_{kl}^*\frac{\partial\psi_0}{\partial x_k}(x)\frac{\partial u^*}{\partial x_l}(x)
\end{align}
Using equations~\eqref{eq1} and~\eqref{eq2}, we obtain
\begin{align}
\psi_0(x)\left(\mathcal{A}^{hom}u^*-f\right)(x)=\frac{\partial\psi_0}{\partial x_k}\left[a_{kl}^*\frac{\partial u^*}{\partial x_l}(x)-\sigma_k^*(x)\right].\end{align}
Let $\omega$ be a unit vector in $\mathbb{R}^d$, then $\psi_0(x)e^{ix\cdot\omega}\in\mathcal{D}(\Omega)$. On substituting in the above equation, we get, for all $k=1,2,\ldots,d$ and for all $\psi_0\in\mathcal{D}(\Omega)$,
\begin{align}
\psi_0(x)\left[a_{kl}^*\frac{\partial u^*}{\partial x_l}(x)-\sigma_k^*(x)\right]=0.
\end{align}
Let $x_0$ be an arbitrary point in $\Omega$ and let $\psi_0(x)$ be equal to $1$ near $x_0$, then for a small neighborhood of $x_0$:
\begin{align}
\mbox{ for } k=1,2,\ldots,d,~\left[a_{kl}^*\frac{\partial u^*}{\partial x_l}(x)-\sigma_k^*(x)\right]=0
\end{align}
However, $x_0\in\Omega$ is arbitrary, so that
\begin{align}
\mathcal{A}^{hom}u^*=f\mbox{ and }\sigma^*_k(x)=a_{kl}^*\frac{\partial u^*}{\partial x_l}(x).
\end{align}Thus,we have obtained the limit equation in the physical space. This finishes the proof of Theorem~\ref{homog}.

\section{Identification of the Homogenized Tensor}\label{section6}

In this section, we recall that $a_{kl}^*$ can be identified with the homogenized tensor for the almost periodic operator $\mathcal{A}^\epsilon$~\cite{Kozlov78,Oleinik1982,Jikov1994} and that $a_{kl}^*$ does not depend on any subsequence of $a^{R,*}_{kl}$. The study of homogenization of almost periodic media was initiated by Kozlov~\cite{Kozlov78} who also obtained a convergence rate for a subclass of quasiperiodic media. Subsequently, an abstract approach which seeks solutions without derivatives was described in~\cite{Oleinik1982,Jikov1994} and is explained in the next subsection.
\subsection{Cell Problem for Almost Periodic Media}\label{APhom}
We begin by introducing the cell problem for almost periodic operator $\mathcal{A}$. Consider the set $S=\{\nabla\phi:\phi\in \Trig(\mathbb{R}^d;\mathbb{R})\}$ as a subset of $(B^2(\mathbb{R}^d))^d$, the Hilbert space of all $d$-tuples of $B^2(\mathbb{R}^d)$ functions. Let $W$ denote the closure of $S$ in $(B^2(\mathbb{R}^d))^d$. Let ${U}=(u_1,u_2,\ldots,u_d)\in W$ and ${V}=(v_1,v_2,\ldots,v_d)\in W$. On $W$, define the bilinear form 
\begin{align}
{\bf a}({U},{V})\coloneqq\mathcal{M}(A{U}\cdot{V}).
\end{align}

Then clearly the bilinear form $\bf a$ is continuous and coercive on $W$. Let $\xi\in\mathbb{R}^d$. Define a linear form on $W$ by
\begin{align}
{\bf l}_\xi({V})\coloneqq -\mathcal{M}(A\xi\cdot{V}),
\end{align} for $V\in W$.
The linear form ${\bf l}_\xi$ is continuous on $W$. As a consequence, by Lax-Milgram lemma, the problem 
\begin{align}\label{cellAP}
{\bf a}(N^\xi,V)={\bf l}_\xi(V),\,\forall V\in W
\end{align} has a solution $N^\xi\in W$ and by the classical theory of almost periodic homogenization~\cite{Oleinik1982}, the homogenized coefficients for $\mathcal{A}^\epsilon$ are given by \begin{align}\label{homoAP}
q_{kl}^*=\mathcal{M}\left(e_k\cdot Ae_l+e_k\cdot AN^{e_l}\right),
\end{align} where $e_i$ denotes the unit vector in $\mathbb{R}^d$ with $1$ in the $i^{th}$ place and $0$ elsewhere.

Since periodic media are also almost periodic, a question arises as to whether the formulation~\eqref{cellAP} is consistent with~\eqref{CellP}.

We restate the two cell problems here in their variational formulations:

The corrector $w^{R,\xi}$ satisfies
\begin{align}\label{Cell1}
\mathcal{M}_{Y_R}\left(A^R\nabla w^{R,\xi}\cdot\nabla\phi\right)=-\mathcal{M}_{Y_R}\left(A^R\xi\cdot\nabla \phi\right)
\end{align} for all $\phi\in H^1_\sharp(Y_R)$ whereas $N^{R,\xi}$ satisfies
\begin{align}\label{Cell2}
\mathcal{M}\left(A^RN^{R,\xi}\cdot V\right)=-\mathcal{M}(A^R\xi\cdot V),
\end{align} for all $V\in W$.

\begin{lemma}\label{APisP}
	Let $w^{R,\xi}$ and $N^{R,\xi}$ satisfy~\eqref{Cell1} and~\eqref{Cell2} respectively, then it holds that $$\displaystyle N^{R,\xi}=\nabla w^{R,\xi}.$$
\end{lemma}  

\begin{proof} We will show that $\nabla w^{R,\xi}$ solves the variational formulation~\eqref{Cell2}. To see this, it is enough to use test functions $V\in S$. Further, due to linearity, it is enough to use test functions of the form $V=\nabla(e^{iy\cdot\eta})$. Now, observe that if $\eta\in 2\pi R\mathbb{Z}^d$, then $\nabla w^{R,\xi}$ satisfies~\eqref{Cell2} since it reduces to equation~\eqref{Cell1} due to the equality $\mathcal{M}(f)=\mathcal{M}_{Y_R}(f)$ for $Y_R$-periodic functions $f$. On the other hand, if $\eta\not\in 2\pi R\mathbb{Z}^d$, once again $\nabla w^{R,\xi}$ satisfies~\eqref{Cell2}, both sides of which are identically zero, because $\mathcal{M}(f(\cdot)e^{iy\cdot\eta})=0$ whenever $\eta$ is not among the frequencies of $f$. Hence, in either case, $\nabla w^{R,\xi}$ satisfies equation~\eqref{Cell2}. Finally, due to uniqueness, \begin{align*}N^{R,\xi}=\nabla w^{R,\xi}.\end{align*}\end{proof}


Given an almost periodic function $f$, let $\Lambda(f)$ denote the set of all $\xi\in\mathbb{R}^d$ such that $\mathcal{M}(fe^{-ix\cdot\xi})\neq 0$.  Let $Mod(f)$ be the $\mathbb{Z}$-module generated by $\Lambda(f)$. The $\mathbb{Z}$-module $Mod(f)$ shall be referred to as the frequency module of $f$. In the argument above, we have shown that $Mod(N^{R,\xi})\subseteq Mod(A^R)$. This argument can be readily generalized to a module containment theorem for the correctors. In particular, we may prove that $Mod(N^\xi)\subseteq Mod(A)$. To paraphrase, the frequencies of the correctors are generated from the frequencies of the coefficients. To this end, we define a closed subspace of the Hilbert space $B^2(\mathbb{R}^d)$ in the following manner. Consider the set of all real trigonometric polynomials whose exponents come from $Mod(A)$ and call it $\Trig_A(\mathbb{R}^d;\mathbb{R})$. The closure of $\Trig_A(\mathbb{R}^d;\mathbb{R})$ in $B^2(\mathbb{R}^d)$ will be denoted by $B^2_A(\mathbb{R}^d)$. Consider the set $S_A=\{\nabla\phi:\phi\in \Trig_A(\mathbb{R}^d)\}$ as a subset of $(B^2_A(\mathbb{R}^d))^d$, the Hilbert space of all $d$-tuples of $B^2_A(\mathbb{R}^d)$ functions. Let $W_A$ denote the closure of $S_A$ in $(B^2(\mathbb{R}^d))^d$. To begin with, we prove that the frequencies of a given function $u\in B^2_A(\mathbb{R}^d)$ belong to $Mod(A)$.

\begin{lemma}
	Let $u\in B^2_A(\mathbb{R}^d)$. Let $\xi\in\mathbb{R}^d$ such that $\mathcal{M}(u\cdot e^{ix\cdot\xi})\neq 0$, then $\xi\in Mod(A)$.
\end{lemma}

\begin{proof}
	Since $u\in B^2_A(\mathbb{R}^d)$, we have a sequence of trigonometric polynomials $u_n\in Trig_A(\mathbb{R}^d)$ such that $\mathcal{M}(|u_n-u|^2)\to 0$. Let $\xi\notin Mod(A)$, then
	\begin{align*}
		|\mathcal{M}(u\cdot e^{ix\cdot\xi})|&\leq |\mathcal{M}(u_n\cdot e^{ix\cdot\xi})|+|\mathcal{M}((u_n-u)\cdot e^{ix\cdot\xi})|\\
		&=|\mathcal{M}((u_n-u)\cdot e^{ix\cdot\xi})|\\
		&\leq \left(\mathcal{M}(|u_n-u|^2)\right)^{1/2},
	\end{align*} which can be made arbitrarily small. Therefore, $\mathcal{M}(u\cdot e^{ix\cdot\xi})=0$.
\end{proof}

Now the equation $$-\dive(A(\xi+ N))=0\mbox{ in }\mathbb{R}^d$$ has two variational formulations as below:

Find $N^\xi_A\in W_A$ such that
\begin{align}\label{Cell3}
\mathcal{M}\left(AN_A^{\xi}\cdot V\right)=-\mathcal{M}(A\xi\cdot V),
\end{align} for all $V\in W_A$ and find $N^\xi\in W$ such that
\begin{align}\label{Cell4}
\mathcal{M}\left(AN^{\xi}\cdot V\right)=-\mathcal{M}(A\xi\cdot V),
\end{align} for all $V\in W$.

\begin{lemma}\label{modulecontainment}
	Let $N_A^{\xi}$ and $N^{\xi}$ satisfy~\eqref{Cell3} and~\eqref{Cell4} respectively, then it holds that $$\displaystyle N^{\xi}=N_A^{\xi}.$$ In particular, $N^\xi\in (B^2_A(\mathbb{R}^d))^d$ and hence $Mod(N^{\xi})\subseteq Mod(A)$.
\end{lemma}  

\begin{proof} We will show that $N_A^{\xi}$ solves the variational formulation~\eqref{Cell4}. To see this, it is enough to use test functions $V\in S$. Further, due to linearity, it is enough to use test functions of the form $V=\nabla(e^{iy\cdot\eta})$. Now, observe that if $\eta\in Mod(A)$, then $N_A^{\xi}$ satisfies~\eqref{Cell4} since it is the same as equation~\eqref{Cell3}. On the other hand, if $\eta\not\in Mod(A)$, once again $N_A^{\xi}$ satisfies~\eqref{Cell4}, both sides of which are identically zero, because $\mathcal{M}(f(\cdot)e^{iy\cdot\eta})=0$ whenever $\eta$ is not among the frequencies of $f$. Hence, in either case, $N_A^{\xi}$ satisfies equation~\eqref{Cell4}. Finally, due to uniqueness, \begin{align*}N^{\xi}=N_A^{\xi}.\end{align*}
\end{proof}

\begin{remark}
	By Lemma~\ref{modulecontainment}, we can conclude that if $A$ is periodic then $N^\xi$ is also periodic. Thus, it is possible to conclude Lemma~\ref{APisP} from Lemma~\ref{modulecontainment}. We would also like to point out that Lemma~\ref{modulecontainment} is a qualitative version of Theorem~\ref{ShenContainment} where the almost periodicity of $\nabla w^\xi$ is expressed in terms of almost periodicity of $A$. Module containment results pertaining to a variety of differential equations may be found in~\cite{Fink74,Amerio1971}.
\end{remark}

\subsection{Convergence of Homogenized Tensors}

It was proved by Bourgeat and Piatnitski~\cite[Theorem~1]{BourgeatPiatnitski2004} that approximate homogenized tensors defined in~\eqref{homoR} using periodic correctors defined in~\eqref{CellP} converge to the homogenized tensor~\eqref{homoAP} of almost periodic media. They rely on homogenization theorem for almost periodic operators~\cite[p.~241]{Jikov1994} and an auxilliary result on convergence of ``arbitrary solutions"~\cite[Theorem~5.2]{Jikov1994}. We restate this theorem here without proof for which we refer to~\cite{BourgeatPiatnitski2004}.

\begin{theorem}(Bourgeat \& Piatnistski~\cite[Theorem~1]{BourgeatPiatnitski2004})
 Let $1\leq k,l\leq d$ and let $a^{R,*}_{kl}$ and $q^*_{kl}$ be defined as in~\eqref{homoR} and~\eqref{homoAP} respectively, then $a^{R,*}_{kl}\to q^*_{kl}$ as $R\to\infty$.
\end{theorem}

In Subsection~\ref{boundednessofcellfunctions}, we showed that the sequence of homogenized tensors $a_{kl}^{R,*}$ is bounded and hence converges for a subsequence to a limit $a_{kl}^*$. The theorem of Bourgeat and Piatnitski shows that, in fact, the whole sequence converges to the limit $q_{kl}^*$. Therefore, $a_{kl}^*=q_{kl}^*$.


\section{Higher modes do not contribute}\label{section7}
The proof of the qualitative homogenization theorem (Theorem~\ref{homog}) only requires the first Bloch transform. It is not clear whether the higher Bloch modes make any contribution to the homogenization limit. In this section, we show that they do not. We know that Bloch decomposition is the isomorphism $L^2(\mathbb{R}^d)\cong  L^2(Y^{'};\ell^2(\mathbb{N}))$ which is reflected in the inverse identity~\eqref{Blochinverse}. For simplicity, take $\Omega=\mathbb{R}^d$ and consider the equation $\mathcal{A}^\epsilon u^\epsilon=f$ in $\mathbb{R}^d$ which is equivalent to \begin{align*}\mathcal{B}^{R,\epsilon}_m \mathcal{A}^\epsilon u^\epsilon(\xi)=\mathcal{B}^{R,\epsilon}_mf(\xi)\quad\forall m\geq 1,\forall\,\xi\in\epsilon^{-1}Y_R^{'},\end{align*} which may be further expanded to
\begin{align*}
\mathcal{B}^{R,\epsilon}_m \mathcal{A}^{R,\epsilon}u^\epsilon(\xi)=\mathcal{B}^{R,\epsilon}_mf(\xi)+\left(\mathcal{B}^{R,\epsilon}_m\nabla\cdot(A^\epsilon-A^{R,\epsilon})\nabla u^\epsilon\right)(\xi)\quad\forall m\geq 1,\forall\,\xi\in\epsilon^{-1}Y_R^{'},
\end{align*} or \begin{align}\label{cascade}
\lambda^{R,\epsilon}_m(\xi)\mathcal{B}^{R,\epsilon}_m u^\epsilon(\xi)=\mathcal{B}^{R,\epsilon}_mf(\xi)+\left(\mathcal{B}^{R,\epsilon}_m\nabla\cdot(A^\epsilon-A^{R,\epsilon})\nabla u^\epsilon\right)(\xi)\quad\forall m\geq 1,\forall\,\xi\in\epsilon^{-1}Y_R^{'}.
\end{align}
We claim that one can neglect all the equations corresponding to $m\geq 2$.

\begin{proposition}
	Let $$v^{R,\epsilon}(x)=\int_{\epsilon^{-1}Y_R^{'}}\sum_{m=2}^{\infty}\mathcal{B}^{R,\epsilon}_m u^\epsilon(\xi)\phi_m^{R,\epsilon}(x;\xi)e^{ix\cdot\xi}\,d\xi,$$ then $||v^{R,\epsilon}||_{L^2(\mathbb{R}^d)}\leq cR\epsilon.$ Hence, given any sequence, $\epsilon_k\to 0$, we can find a sequence $R_{k}$ such that $v^{R_k,\epsilon_k}\to 0$ as $k \to \infty$.
\end{proposition}
\begin{proof}
	Due to boundedness of the sequence $(u^\epsilon)$ in $H^1(\mathbb{R}^d)$, we have
	\begin{align}
	\int_{\mathbb{R}^d}\mathcal{A}^{R,\epsilon} u^\epsilon\,\overline{u^\epsilon}\leq C.
	\end{align}
	However, by Plancherel Theorem~\eqref{Plancherel}, we have
	\begin{align*}
	\int_{\mathbb{R}^d} \mathcal{A}^{R,\epsilon} u^\epsilon\, \overline{u^\epsilon}=\sum_{m=1}^{\infty}\int_{\epsilon^{-1}Y_R^{'}}\left(\mathcal{B}^{R,\epsilon}_m\mathcal{A}^{R,\epsilon} u^\epsilon\right)(\xi)\,\overline{\mathcal{B}^{R,\epsilon}_mu^\epsilon(\xi)}\,d\xi\leq C
	\end{align*}
	Using~\eqref{diagonalizationepsilon}, we have
	\begin{align*}
	\sum_{m=1}^{\infty}\int_{\epsilon^{-1}Y_R^{'}}\lambda^{R,\epsilon}_m(\xi)|{\mathcal{B}^{R,\epsilon}_mu^\epsilon(\xi)}|^2\,d\xi\leq C.
	\end{align*}
	Now, by a simple application of Courant-Fischer min-max principle, we can show that
	\begin{align}
	\lambda_m^R(\eta)\geq\lambda_2^R(\eta)\geq\lambda^R_2(-\Delta,R)\geq\frac{C}{R^2}>0\quad\forall\, m\geq 2\quad\forall\,\eta\in Y^{'}_R,
	\end{align} where $\lambda^R_2(-\Delta,R)$ is the second eigenvalue of Laplacian on $Y_R$ with Neumann boundary condition on $\partial Y_R$. The bound quoted is standard for the Neumann Laplacian on a rectangle but it may also be understood as an instance of the fundamental gap inequality for Neumann Laplacian on convex domains~\cite{Payne1960}. We also know that $\lambda^{R,\epsilon}_m(\xi)=\epsilon^{-2}\lambda^{R,\epsilon}_m$, therefore, combining these two facts, we obtain 
	\begin{align*}
	\sum_{m=2}^{\infty}\int_{\epsilon^{-1}Y_R^{'}}|{\mathcal{B}^{R,\epsilon}_mu^\epsilon(\xi)}|^2\,d\xi\leq CR^2\epsilon^2.
	\end{align*}
	Now, given any sequence $\epsilon_k\to 0$, we can choose a sequence $R_{k}\to\infty$ such that $R_{k}^2<\frac{1}{\epsilon}$, then along this sequence
	\begin{align*}
	\sum_{m=2}^{\infty}\int_{\epsilon^{-1}Y_R^{'}}|{\mathcal{B}^{R,\epsilon}_mu^\epsilon(\xi)}|^2\,d\xi\leq C\epsilon.
	\end{align*}
	By Parseval's identity, the left side is equal to $||v^{R,\epsilon}||^2_{L^2(\mathbb{R}^d)}$. This completes the proof of this Proposition.
\end{proof}

\begin{remark}
	The product $``R\epsilon"$ is the resonance error~\cite{Gloria2011},~\cite{GloriaHabibi2016} due to the approximation. The above discussion explains the relation between higher modes of the Bloch spectrum and the resonance error. In particular, the periodic approximation serves to separate the lowest Bloch mode from the rest of the spectrum and the limit $R\to\infty$ represents the loss of simplicity and hence analyticity of the lowest Bloch eigenvalue near zero.
\end{remark}

\section{Rate of convergence for approximations}\label{roc}

Let $a^*_{kl}$ denote the $(k,l)^{th}$ entry of the homogenized tensor for the almost periodic operator. We shall mostly write this as $e_k\cdot A^*e_l$. Similarly, the homogenized tensor associated to the periodization $A^R$ will be denoted by $A^{R,*}$. In Section~\ref{section6}, we observed that $A^{R,*}\to A^*$. In this section, we shall obtain a rate of convergence estimate for the error $|A^*-A^{R,*}|$.

\subsection{Volume Averaging Method}\label{subsection1} In engineering, the Volume averaging method~\cite{Whitaker1988}  is employed to determine effective behavior of heterogeneous media by using averages of physical quantities, such as energy, on a large volume of the domain under consideration, called a Representative Elementary Volume~\cite{whitaker2013method}. A comparison between the mathematical theory of homogenization and volume averaging is carried out in~\cite{davit2013homogenization}. In a well-known paper of Bourgeat and Piatnitski~\cite{BourgeatPiatnitski2004}, the volume averaging technique has been employed to obtain approximations to homogenized tensor for stochastic media. 

The homogenization of stochastic as well as almost periodic media has two major difficulties - the cell problem is posed on $\mathbb{R}^d$ and the loss of differential structure, i.e., the correctors do not appear as derivatives in the cell problem~\eqref{cellAP} in almost periodic and stochastic homogenization. The differential structure is important as it is responsible for the compensated compactness of the oscillating test functions in homogenization~\cite{cherkaev97}. As a compromise, many authors such as Kozlov~\cite{kozlov}, Yurinski~\cite{Yurinski1986} have introduced cell problems with a penalization (or regularization) term to recover the differential structure. However, these problems are still posed on $\mathbb{R}^d$. The homogenized tensor appears as a mean value on $\mathbb{R}^d$ which makes the computation of homogenized tensor impossible. Hence, volume averages on large cubes provide a suitable proxy for the homogenized tensor.

Like stochastic media, almost periodic media exhibit {\it long range order}. Stochastic media is quantified in terms of mixing coefficients. In contrast, the process generated by almost periodic media is not mixing, although it is ergodic~\cite{Simon1982}. In some sense, almost periodic functions fall half way between periodic and random media. Therefore, a quantification specific to almost periodicity is required in order to obtain quantitative results in homogenization theory. A modulus $\rho(A)$ of almost periodicity is defined in~\cite{Armstrong2014}, which has been employed by Shen~\cite{Shen2015} to extend the compactness methods in~\cite{AvellanedaLin1987} to almost periodic homogenization. Shen also proves that the small divisors condition of Kozlov~\cite{Kozlov78} implies a decay hypothesis on $\rho(A)$. Kozlov was the first to prove a rate of convergence estimate in homogenization of almost periodic media satisfying the small divisors condition. Thereafter, quantitative homogenization of almost periodic operators has seen a resurgence in the works of Armstrong, Shen and coauthors~\cite{Armstrong2014, Shen2015,ShenZhuge2018,ArmstrongShen2016,ArmstrongGloriaKuusi2016}. 

\subsection{Rate of convergence estimates}  In this subsection, we will estimate the error $|A^*-A^{R,*}|$ using the strategy of Bourgeat and Piatnitski~\cite{BourgeatPiatnitski2004}. Their techniques were refined and improved by Gloria and his coauthors~\cite{Gloria2011},~\cite{GloriaHabibi2016},~\cite{GloriaOtto2017}. We shall follow the ideas of these authors to establish convergence rate for $A^{R,*}$ in terms of the following quantification of almost periodicity as introduced in~\cite{Armstrong2014}. For a matrix $A$ with continuous and bounded entries, define the following modulus of almost periodicity:
\begin{align}\label{modulusAP}
\rho(A,L)\coloneqq \sup_{y\in\mathbb{R}^d}\inf_{|z|\leq L}||A(\cdot+y)-A(\cdot+z)||_{L^{\infty}(\mathbb{R}^d)}.
\end{align}
It follows that $A$ is almost periodic if and only if $\rho(A,L)\to 0$ as $L\to\infty$. In particular, for periodic functions, the modulus becomes zero for large $L$. We are now ready to state the theorem on the rate of convergence.

\begin{theorem}\label{THEOREM}If $A\in AP(\mathbb{R}^d)$ is such that, for each $L>0$, $\rho(A,L)$ satisfies $\rho(A,L)\lesssim 1/L^\tau$ for some $\tau>0$, then, there exists a $\beta\in(0,1)$ such that	\begin{align}
	|A^*-A^{R,*}|\lesssim \frac{1}{{R}^{\beta}},
	\end{align} where $A^*$ and $A^{R,*}$ are defined in~\eqref{homoAP} and~\eqref{homoR} respectively.\qed 
\end{theorem}

%
%
%
%


\subsection{Strategy of Proof}  The proof of Theorem~\ref{THEOREM} will be done in four steps. We have already seen two cell problems corresponding to the almost periodic media and its Periodic approximation, viz.,~\eqref{cellAP} and~\eqref{CellP}. We shall require two more cell problems, corresponding to regularization of~\eqref{cellAP} and~\eqref{CellP}. For the sake of convenience, we list all the requisite cell problems below.   For $\xi\in\mathbb{R}^d$ and $T>0$:

\begin{itemize}
	\item[\bf (P)] Find $w^{R,\xi}\in H^1_{\sharp}(Y_R)$ such that 
	\begin{align}\label{cell4}
	-\nabla\cdot(A (\xi+\nabla w^{R,\xi}))=0.
	\end{align}
	\item[\bf (PT)] Find $ w_{T}^{R,\xi}\in H^1_\sharp(Y_R)$ such that 
	\begin{align}\label{cell3}
	-\nabla\cdot(A (\xi+\nabla w_{T}^{R,\xi}))+T^{-1}w_{T}^{R,\xi}=0.
	\end{align}
\end{itemize}
\begin{itemize}
	\item[\bf (AP)] Find $N^{\xi}\in (B^2(\mathbb{R}^d))^d$ such that 
	\begin{align}\label{cell1}
	\mathcal{M}\left(A N^{\xi}\cdot v\right)=-\mathcal{M}(A\xi\cdot v)
	\end{align} for all $v\in\{\nabla\phi:\phi\in \Trig(\mathbb{R}^d)\}$.
	\item[\bf (APT)] Find $ w_{T}^{\xi}\in H^1_{\loc}(\mathbb{R}^d)$ such that 
	\begin{align}\label{cell2}
	-\nabla\cdot(A (\xi+\nabla w_{T}^{\xi}))+T^{-1}w_{T}^{\xi}=0.
	\end{align}
\end{itemize}



The homogenized tensor $A^*$ is defined as $$\xi\cdot A^*\xi=\mathcal{M}\left((\xi+N^{\xi})\cdot A(\xi+N^{\xi})\right).$$ Define $A^*_T$ as \begin{align}\label{homoAPT}\xi\cdot A^*_T\xi=\mathcal{M}\left((\xi+\nabla w_{T}^{\xi})\cdot A(\xi+\nabla w_{T}^{\xi})\right).\end{align} Also, define the truncated average $\overline{A}_{T,R}$ as $$\xi\cdot \overline{A}_{T,R}\xi=\frac{1}{|Y_R|}\int_{Y_R}\left((\xi+\nabla w_{T}^{\xi})\cdot A(\xi+\nabla w_{T}^{\xi})\right)~dy,$$ and  define the Periodic approximation $A^{R,*}$ to $A^*$ as $$\xi\cdot A^{R,*}\xi=\frac{1}{|Y_R|}\int_{Y_R}\left((\xi+\nabla w^{R,\xi})\cdot A(\xi+\nabla w^{R,\xi})\right)~dy.$$ The homogenized tensor corresponding to the regularized Periodic cell problem~\eqref{cell3} is \begin{align}\label{homoAPTD}\xi\cdot A^{R,*}_{T}\xi=\frac{1}{|Y_R|}\int_{Y_R}\left((\xi+\nabla w_{T}^{R,\xi})\cdot A(\xi+\nabla w_{T}^{R,\xi})\right)~dy.\end{align}

With the notation in place, we can proceed with the strategy for obtaining the rate of convergence estimates. This is essentially the same as the one employed by Bourgeat and Piatnitski~\cite{BourgeatPiatnitski2004} to obtain estimates for approximations of homogenized tensor for random ergodic media. We shall write \begin{align}\label{fourparts}|A^*-A^{R,*}|\leq|A^*-A^*_T|+|A^*_T-\overline{A}_{T,R}|+|\overline{A}_{T,R}-A^{R,*}_{T}|+|A^{R,*}_{T}-A^{R,*}|\end{align}

In the above inequality, the first and last terms on RHS are estimated in terms of the rate of convergence of regularized correctors to the exact correctors as $T\to\infty$. The proof of this estimate for the first term is available in Shen~\cite{Shen2015}. For the proof of estimate for the last term, we adapt the argument in Bourgeat and Piatnitski~\cite{BourgeatPiatnitski2004}.

The second term corresponds to rate of convergence in mean ergodic theorems. This estimate is available for periodic and quasiperiodic functions and is of order $1/R$. In Blanc and Le Bris~\cite{BlancLeBris2010} and Gloria~\cite{Gloria2011}, a different truncated approximation is proposed, through the use of filters; either as a weight in the cell problem or as post-processing. Such approximations have faster rates of convergence. However, we shall write this rate of convergence in terms of $\rho(A,L)$ following~\cite{ShenZhuge2018}.

The third term on RHS corresponds to a boundary term which is controlled by the Green's function decay of the regularized operator $T^{-1}-\nabla\cdot(A\nabla)$ in $\mathbb{R}^d$. The proof is essentially due to Bourgeat and Piatnitski~\cite{BourgeatPiatnitski2004} but has lately been refined by Gloria~\cite{Gloria2011} (also see~\cite{GloriaOtto2017}).

In the next subsections, we shall prove the four convergence rates.

\subsection{Rate of convergence of regularized correctors}
We will begin by establishing the existence of the regularized correctors as defined in~\eqref{cell2}. This can be done in two ways. One is by following the derivation theory of Besicovitch spaces as presented in Casado-D\'{i}az and Gayte~\cite{CasadoDiaz2002}. The other method is to build solutions in $H^1_{\loc}(\mathbb{R}^d)$ directly by approximations on disks~\cite{Yurinski1989},~\cite{Shen2015}. The second method is more general as it does not require the assumption of almost periodicity on the coefficients. However, the existence of a derivation theory on Besicovitch spaces makes it easier to obtain a priori estimates. 

For $p\in(1,\infty)$, ${B}^p(\mathbb{R}^d)$ is the closure of trigonometric polynomials in the semi-norm $\mathcal{M}(|\cdot|^p))^{1/p}$. Let $D^\infty$ be the space
\begin{align}
D^{\infty}\coloneqq\{\,\phi\in C^\infty(\mathbb{R}^d)\,:\,D^\alpha\phi\in B^1(\mathbb{R}^d)\cap L^\infty(\mathbb{R}^d)\mbox{ for all multiindices } \alpha \,\},
\end{align} which is analogous to the space of test functions for defining weak derivatives in the theory of distributions. Next, given a function $u\in B^1(\mathbb{R}^d)$, define its mean derivative $\partial_{j}u$ as a linear map on $D^\infty$ given by $\partial_{j}u(\phi)\coloneqq-\mathcal{M}\left(u\frac{\partial \phi}{\partial x_j}\right)$. This definition is well defined in the sense that if $u_1$ and $u_2$ are two functions in $B^1(\mathbb{R}^d)$ such that $\mathcal{M}(|u_1-u_2|)=0$, then they define the same mean derivative. Moreover, if the distributional derivative of a function $u\in B^1(\mathbb{R}^d)$ is also in $B^1(\mathbb{R}^d)$, then it agrees with the mean derivative of $u$. The following definition of the Besicovitch analogue of Sobolev spaces is presented in~\cite{CasadoDiaz2002}:
\begin{align}
B^{1,p}(\mathbb{R}^d)\coloneqq\{\,u\in B^p(\mathbb{R}^d)\,:\,\exists~u_j\in B^p(\mathbb{R}^d) \mbox{ such that } \partial_{j}u(\phi)=\mathcal{M}(u_j\phi),~1\leq j\leq d\,\}.	
\end{align} This space admits the semi-norm 
\begin{align*}
	|u|_{\mathcal{M}}=\mathcal{M}(|u|)+\mathcal{M}(|\nabla u|).
\end{align*}
It can be made into a Banach space by identifying those elements whose difference has zero semi-norm. We shall continue to denote the associated Banach space as $B^{1,p}(\mathbb{R}^d)$. Further, every representative $u$ is an element of $W^{1,p}_{\loc}(\mathbb{R}^d)$ with the property that any two representatives $u_1$ and $u_2$ satisfy $|u_1-u_2|_\mathcal{M}=0$.

\begin{theorem}
	Let the matrix $A$ satisfy~\ref{A1},~\ref{A2},~\ref{A3}. Then equation~\eqref{cell2} has a unique solution $w_{T}^{\xi}\in B^{1,2}(\mathbb{R}^d)$, and
	\begin{align}\label{boundonpenalized}
	T^{-1}\mathcal{M}(|w_{T}^{\xi}|^2)+\mathcal{M}(|\nabla w_{T}^{\xi}|^2)\lesssim 1.
	\end{align}
\end{theorem}

\begin{proof}
	The space $B^{1,2}(\mathbb{R}^d)$ is a Hilbert space. Define the bilinear form 
	\begin{align*}
	a(w,v)\coloneqq\mathcal{M}(A\nabla w\nabla v+T^{-1}wv),
	\end{align*} which is elliptic due to coercivity of $A$. Also, define the linear form 
	\begin{align*}
	l(v)\coloneqq -\mathcal{M}(A\xi\cdot\nabla v)
	\end{align*} for $v\in B^{1,2}(\mathbb{R}^d)$. The equation~\eqref{cell2} is said to have a solution in $B^{1,2}(\mathbb{R}^d)$ if there exists $w_{T}^{\xi}\in B^{1,2}(\mathbb{R}^d)$ such that $a(w_{T}^{\xi},v)=l(v)$ for all $v\in D^\infty$. The existence and uniqueness of such a solution is guaranteed by an application of Lax-Milgram lemma. Each representative of $w_T^{\xi}\in B^{1,2}(\mathbb{R}^d)$ is an element of $H^1_{\loc}(\mathbb{R}^d)$. The estimate~\eqref{boundonpenalized} is obtained from the weak formulation by choosing $v=w_{T}^{\xi}$ followed by an application of Young's inequality.
\end{proof}

The convergence rate for the first term in~\eqref{fourparts} is available in Shen~\cite{Shen2015} in terms of the function $\rho(A,\cdot)$.

\begin{theorem}[(Shen~\cite{Shen2015}, Remark~6.7)]\label{convergence1} Let $\rho(A,L)$ satisfy $\rho(A,L)\lesssim 1/L^\tau$ for some $\tau>0$. Then for any $\omega$ such that $0<\omega<1$, 
	\begin{align}|A^*-A^*_T|\leq C {{T}^{-\frac{\tau}{2(\tau+1)}+\omega}},\end{align}	
	where the constant is independent of $T$ but depends on $\omega$. 
\end{theorem}

\subsection{Rate of convergence of truncated homogenized tensor}

In proving the convergence of truncated averages $\overline{A}_{T,R}$ to $A^*_T$, we need to show that the almost periodicity of the correctors $w_{T}^{\xi}$ can be quantified in terms of the almost periodicity of $A$. This is the content of the following theorem from Shen~\cite{Shen2015}.

\begin{theorem}(Shen~\cite{Shen2015}, Lemma 5.3)\label{ShenContainment}
	For $y,z\in\mathbb{R}^d$, the regularized corrector $w_{T}^{\xi}$ satisfies \begin{align}
	\left(\dashint_{Y_{R}}|\nabla w_{T}^{\xi}(t+y)-\nabla w_{T}^{\xi}(t+z)|^2~dt\right)^{1/2}\leq C||A(\cdot+y)-A(\cdot+z)||_{L^\infty(\mathbb{R}^d)},
	\end{align} where $C$ is independent of $R, y$, and $z$.
\end{theorem}

Further, Shen and Zhuge~\cite{ShenZhuge2018} have quantified the convergence of truncated averages in terms of almost periodicity of the integrands in the following theorem.

\begin{theorem}(Shen \& Zhuge~\cite{ShenZhuge2018})\label{ergodictheorem}
	For $1<p<\infty$ let $u\in B^p(\mathbb{R}^d)$ and for $p=\infty$ let $u\in AP(\mathbb{R}^d)$. Then for any $0<L\leq R<\infty$,
	\begin{align}
	\left|\dashint_{Y_R}u\,dy-\mathcal{M}(u)\right|\lesssim  \sup_{y\in\mathbb{R}^d}\inf_{|z|\leq L}\dashint_{Y_R}|u(t+y)-u(t+z)|~dt+\left(\frac{L}{R}\right)^{1/p'}\begin{cases}
	||u||_{B^p} & \text{if } p<\infty\\
	||u||_{L^\infty}           & \text{if } p=\infty
	\end{cases}
	\end{align}
\end{theorem}

As a consequence of the two theorems stated above, we can prove the rate of convergence estimate $|\overline{A}_{T,R}-A^*_T|$.

\begin{theorem}\label{convergence2}
	Let $\rho(A,L)$ satisfy $\rho(A,L)\lesssim 1/L^\tau$ for some $\tau>0$. Then for any $0<L\leq R<\infty$,
	\begin{align}\label{estimate2}|\overline{A}_{T,R}-A^*_T|\lesssim \frac{1}{L^\tau}+\left(\frac{L}{R}\right)^{1/2}.\end{align}	 
\end{theorem}

\begin{proof}
	We shall apply Theorem~\ref{ergodictheorem} to the functions $u_1=e_k\cdot Ae_l$ and $u_2=e_k\cdot A\nabla w_{T}^{e_l}$. For $u_1$, we may choose $p=\infty$ to obtain the following estimate.
	\begin{align}\label{something1}|\mathcal{M}(A)-\mathcal{M}_{Y_R}(A)|\lesssim \rho(A,L)+\frac{L}{R}\lesssim\frac{1}{L^\tau}+\frac{L}{R}.\end{align}	
	For $u_2$, we may choose $p=2$. By Theorem~\ref{ergodictheorem}, we have 
	\begin{align}\label{something2}|
	\mathcal{M}(A\nabla w_{T}^{e_l})-\mathcal{M}_{Y_R}(A\nabla w_{T}^{e_l})|&\lesssim \sup_{y\in\mathbb{R}^d}\inf_{|z|\leq L}\dashint_{Y_{R}}|(A\nabla w_{T}^{e_l})(t+y)-(A\nabla w_{T}^{e_l})(t+z)|~dt\nonumber\\
	&\qquad+\left(\frac{L}{R}\right)^{1/2}||u||_{B^2}
	\end{align}
	Through an application of Theorem~\ref{ShenContainment}, we note that
	\begin{align}\label{something3}
	\dashint_{Y_{R}}|(A\nabla w_{T}^{e_l})(t+y)-(A\nabla w_{T}^{e_l})(t+z)|~dt&\leq C||A(\cdot+y)-A(\cdot+z)||_{L^\infty(\mathbb{R}^d)}.
	\end{align}
	Combining~\eqref{something2} and~\eqref{something3}, we get
	\begin{align}\label{something4}|\mathcal{M}(A\nabla w_{T}^{e_l})-\mathcal{M}_{Y_R}(A\nabla w_{T}^{e_l})|&\lesssim \sup_{y\in\mathbb{R}^d}\inf_{|z|\leq L}||A(\cdot+y)-A(\cdot+z)||_{L^\infty(\mathbb{R}^d)}\nonumber\\
    &\qquad+\left(\frac{L}{R}\right)^{1/2}||u||_{B^2}\nonumber\\
    &\lesssim\frac{1}{L^\tau}+\left(\frac{L}{R}\right)^{1/2}.
\end{align}
	Combining~\eqref{something1} and~\eqref{something4}, we get~\eqref{estimate2}.
\end{proof}

\subsection{Rate of convergence of boundary term}

Now, we shall prove estimate on the boundary term, viz., $|\overline{A}_{T,R}-A^{R,*}_{T}|$. The proof is essentially the same as in~\cite{BourgeatPiatnitski2004}, although the Green's function estimates are borrowed from~\cite{GloriaOtto2017}. We begin by recalling the existence of Green's function associated with the operator $T^{-1}-\nabla\cdot(A\nabla)$ and its pointwise bounds.

\begin{theorem}(Gloria \& Otto~\cite{GloriaOtto2017})\label{Greenestimates}
	Let $A$ be a coercive matrix with measurable and bounded entries, and let $T>0$. Then for all $y\in\mathbb{R}^d$, there is a function $G_T(\cdot,y)$ which is the unique solution in $W^{1,1}(\mathbb{R}^d)$ of the equation\begin{align}
	T^{-1}G_T(x,y)-\nabla_x\cdot(A\nabla_xG_T(x,y))=\delta(x-y),
	\end{align}in the sense of distributions. The function $G_T(\cdot,y)$ is continuous on $\mathbb{R}^d\setminus\{y\}$. Furthermore, the Green's function satisfies the following pointwise bounds:
	\begin{align}
	0\leq G_T(x,y)\lesssim\exp\left(-c\frac{|x-y|}{\sqrt{T}}\right)\begin{cases}
	\ln\left(2+\frac{\sqrt{T}}{|x-y|}\right)& \text{if } d=2\\
	{|x-y|^{2-d}},              & \text{if } d>2
	\end{cases}.
	\end{align}
\end{theorem}

\begin{theorem}\label{convergence3}Let $0<\delta<1$, $|\overline{A}_{T,R}-A^{R,*}_{T}|\lesssim R^{(\delta-1)/2}+ \exp\left(-c\frac{R^\delta}{\sqrt{T}}\right)\begin{cases}
	R^d&d>2\\
	R^3&d=2.
	\end{cases}$
\end{theorem}

\begin{proof}
	Let $R\geq R_0>0$ and $\delta\geq \delta_0>0$. The proof will be done in three steps: first to obtain an interior estimate in $Y_{R-R^\delta}$, second to obtain an estimate for the boundary layer $Y_R\setminus Y_{R-R^\delta}$ and the final step to obtain the required convergence rate.
	
	{\bf Step 1.} $ w_{T}^{\xi}$ satisfies the following equation in $\mathbb{R}^d$:
	\begin{align*}
	-\nabla\cdot(A (\xi+\nabla w_{T}^{\xi}))+T^{-1}w_{T}^{\xi}=0.
	\end{align*}
	$ w^{R,\xi}_{T}$ satisfies the following equation in $Y_R$:
	\begin{align*}
	-\nabla\cdot(A (\xi+\nabla w_{T}^{R,\xi}))+T^{-1}w_{T}^{R,\xi}=0.
	\end{align*}
	Hence, their difference satisfies
	\begin{align*}
	T^{-1}(w_{T}^{\xi}-w_{T}^{R,\xi})-\nabla\cdot(A \nabla (w_{T}^{\xi}-w_{T}^{R,\xi}))=0 \mbox{  in } Y_{R}
	\end{align*} in the sense of distributions.
	Set $\phi_1=\chi (w_{T}^{\xi}-w_{T}^{\xi,R})$, where $\chi\in C^{\infty}(\overline{Y_R};\mathbb{R}^+)$, so that $\chi|_{\partial Y_R}=1$, $\chi|_{Y_{R-R^\delta/2}}=0$ and $|\nabla\chi|\lesssim 1/R$. 
	
	Therefore, by the bounds~\eqref{boundonpenalized}, $||\phi_1||^2_{L^2(Y_R)}\lesssim R^dT$ and $||\nabla\phi_1||^2_{L^2(Y_R)}\lesssim R^d$ for $R\lesssim T\lesssim R^2$.
	
	Now, define $\phi_2=w_{T}^{\xi}-w_{T}^{R,\xi}-\phi_1$, then $\phi_2$ satisfies the following equation:
	\begin{align*}
	T^{-1}\phi_2-\nabla\cdot A\nabla\phi_2=-T^{-1}\phi_1+\nabla\cdot A\nabla\phi_1\mbox{ in }Y_R\\
	\phi_2=0\mbox{ on }\partial Y_R.
	\end{align*}
	Hence, we may write 
	\begin{align}
	\phi_2(x)=-\int_{Y_R}T^{-1}\phi_1(y)G_{T,R}(x,y)+A(y)\nabla\phi_1(y)\cdot\nabla G_{T,R}(x,y)~dy,
	\end{align}where $G_{T,R}$ is the Green's function for the operator $T^{-1}-\nabla\cdot A\nabla$ on $Y_R$ with zero Dirichlet boundary conditions, i.e.,
	\begin{align}\label{GreenR}
	T^{-1}G_{T,R}(x,y)-\nabla_x\cdot(A\nabla_xG_{T,R}(x,y))&=\delta(x-y)&\mbox{ in } Y_R\nonumber\\
	G_{T,R}(x,y)&=0&\mbox{ on } \partial Y_R
	\end{align}in the sense of distributions.
	Therefore,
	\begin{align*}
	|\phi_2(x)|&\leq||\phi_1||_{L^2(Y_R)}\left(T^{-1}\int_{Y_R\setminus Y_{R-R^\delta/2}}G_{T,R}^2(x,y)~dy\right)^{1/2}\\
	&\qquad+||A||_{L^\infty}||\nabla\phi_1||_{L^2(Y_R)}\left(\int_{Y_R\setminus Y_{R-R^\delta/2}}|\nabla G_{T,R}(x,y)|^2~dy\right)^{1/2}.
	\end{align*}
	In the above inequality, the second term will be handled by using Caccioppoli's inequality. In particular, let us multiply the equation~\eqref{GreenR} for Green's function $G_{T,R}$ by $\eta^2 G_{T,R}$ (where $\eta$ is to be chosen later) and integrate by parts to obtain:
	\begin{align*}
	0&=T^{-1}\int_{Y_R}\eta^2(y)G_{T,R}^2(x,y)~dy+\int_{Y_R}A(y)\nabla(\eta^2(y)G_{T,R}(x,y))\cdot\nabla G_{T,R}(x,y)~dy\\
	&=T^{-1}\int_{Y_R}\eta^2(y)G_{T,R}^2(x,y)~dy+\int_{Y_R}A(y)\nabla(\eta(y)G_{T,R}(x,y))\cdot\nabla(\eta(y)G_{T,R}(x,y))~dy\\
	&\qquad-\int_{Y_R}G_{T,R}^2(x,y)A(y)\nabla\eta(y)\cdot\nabla\eta(y),
	\end{align*}given that $\eta$ is zero in some neighborhood of $0$. From the last equality, we obtain
	\begin{align*}
	\int_{Y_R}|\nabla(\eta~G_{T,R})|^2~dy\lesssim\int_{Y_R}G_{T,R}^2|\nabla\eta|^2~dy.
	\end{align*}
	Choose the function $\eta\in C^\infty(Y_R,\mathbb{R}_+)$, such that 
	\begin{align}
      \begin{cases}
      \eta=0&\mbox{ in }Y_{R-3R^\delta/4},\\
      \eta=1&\mbox{ in }Y_R\setminus Y_{R-R^\delta/2},\\
      |\nabla\eta|\lesssim 1/R&,\\
      \end{cases}
	\end{align}
	 then the preceding inequality becomes 
	\begin{align*}
	\int_{Y_R\setminus Y_{R-R^\delta/2}}|\nabla G_{T,R}|^2~dy\lesssim\frac{1}{R^2}\int_{Y_R\setminus Y_{R-3R^\delta/4}}G_{T,R}^2~dy.
	\end{align*}Therefore, for all $x\in Y_R$, we have
	\begin{align*}
	|\phi_2(x)|&\lesssim||\phi_1||_{L^2(Y_R)}\left(T^{-1}\int_{Y_R\setminus Y_{R-R^\delta/2}}G_{T,R}^2(x,y)~dy\right)^{1/2}
	\\&\qquad+||\nabla\phi_1||_{L^2(Y_R)}\left(\int_{Y_R\setminus Y_{R-3R^\delta/4}}R^{-2}G^2_{T,R}(x,y)~dy\right)^{1/2}.
	\end{align*}
	For $x\in Y_{R-5R^\delta/6}$, and $y\in Y_R\setminus Y_{R-R^\delta/2}$, we have $||x-y||_\infty\geq||y||_\infty-||x||_\infty\geq R-R^\delta/2-R+5R^\delta/6=R^\delta/3$. Therefore, $|x-y|\gtrsim R^\delta$ Further, note that due to maximum principle, $0\leq G_{T,R}\leq G_T$. Hence, on using the pointwise estimate for $G_T$ (Theorem~\ref{Greenestimates}), the above inequality becomes for $d>2$ and for $x\in Y_{R-5R^\delta/6}$:
	\begin{align*}
	|\phi_2(x)|&\lesssim||\phi_1||_{L^2(Y_R)}\left(T^{-1}\int_{Y_R\setminus Y_{R-R^\delta/2}}G_{T}^2(x,y)~dy\right)^{1/2}\\
	&\qquad+||\nabla\phi_1||_{L^2(Y_R)}\left(\int_{Y_R\setminus Y_{R-3R^\delta/4}}R^{-2}G_{T}^2(x,y)~dy\right)^{1/2}\\
	&\lesssim R^{d/2}R^{2\delta-d\delta}\exp\left(-c\frac{R^\delta}{\sqrt{T}}\right)R^{d/2}+R^{d/2}R^{-1}R^{2\delta-d\delta}\exp\left(-c\frac{R^\delta}{\sqrt{T}}\right)R^{d/2}\\
	&\lesssim R^{d-2\delta+d\delta}\exp\left(-c\frac{R^\delta}{\sqrt{T}}\right).
	\end{align*} in the regime $R\lesssim T\lesssim R^2$. Similar calculations provide the estimate for $d=2$.
	
	Hence, 
	\begin{align*}
	\left(\int_{Y_{R-5R^\delta/6}}|\phi_2(x)|^2~dx\right)^{1/2}\lesssim {R}^{d+2\delta-d\delta} R^{d/2}\exp\left(-c\frac{R^\delta}{\sqrt{T}}\right). 
	\end{align*}
	
	Finally, by an application of Caccioppoli's inequality, we have
	\begin{align*}
	\left(\int_{Y_{R-R^\delta}}|\nabla \phi_2(x)|^2~dx\right)^{1/2}\lesssim R^{d/2}R^{d+3\delta-d\delta}\exp\left(-c\frac{R^\delta}{\sqrt{T}}\right). 
	\end{align*}
	Therefore,
	\begin{align*}
	\left(\int_{Y_{R-R^\delta}}|\nabla(w_{T}^{\xi}(x)-w_{T}^{R,\xi}(x))|^2~dx\right)^{1/2}\lesssim R^{d/2}R^{d+3\delta-d\delta}\exp\left(-c\frac{R^\delta}{\sqrt{T}}\right).
	\end{align*} Thus,
	\begin{align}\label{interiorestimate}
    \left(\frac{1}{R^d}\int_{Y_{R-R^\delta}}|\nabla(w_{T}^{\xi}(x)-w_{T}^{R,\xi}(x))|^2~dx\right)^{1/2}\lesssim R^{d+3\delta-d\delta}\exp\left(-c\frac{R^\delta}{\sqrt{T}}\right). 
	\end{align}
	
	{\bf Step 2}
	Let $\xi=e_l$ and denote the solutions of equations~\eqref{cell2} and~\eqref{cell3} as $w_T^l$ and $w_{T}^{R,l}$. For $x\in Y_1$, define the functions
	\begin{align*}
	\tilde{w}^l_T(x)&=\frac{1}{R}w^l_T(Rx)\\
	\tilde{w}^{R,l}_{T}(x)&=\frac{1}{R}w^{R,l}_{T}(Rx).
	\end{align*}
	Then these functions satisfy respectively the following equations in $Y_1$:
	\begin{align*}
	-\nabla\cdot(A\nabla \tilde{w}^l_T)+R^2T^{-1}\tilde{w}^l_T&=\nabla Ae_l,\\
	-\nabla\cdot(A\nabla \tilde{w}^{R,l}_{T})+R^2T^{-1}\tilde{w}^{R,l}_{T}&=\nabla Ae_l.
	\end{align*}
	Also,
	\begin{align}\label{boundaryestimate}
	\begin{rcases}
	\int_{Y_1}|\nabla\tilde{w}^l_{T}(x)|^2~dx&\lesssim\frac{1}{R^d}\int_{Y_R}|\nabla{w}^l_{T}(x)|^2~dx\lesssim C,\\
	\int_{Y_1}|\nabla\tilde{w}^{R,l}_{T}(x)|^2~dx&\lesssim\frac{1}{R^d}\int_{Y_R}|\nabla{w}^{R,l}_{T}(x)|^2~dx\lesssim C,
	\end{rcases}
	\end{align} where $C$ is a generic constant.
	Now, we can obtain the required estimates.
	
	{\bf Step 3.} On using~\eqref{interiorestimate} and~\eqref{boundaryestimate}, we have
	\begingroup
		\allowdisplaybreaks
	\begin{align*}
	|&e_k\cdot(\overline{A}_{T,R}-A^{R,*}_{T})e_l|\\
	&=\left|\dashint_{Y_R} e_k\cdot A\nabla({w}^l_{T}-{w}^{R,l}_{T})~dx \right|\\
	&\lesssim\left|\frac{1}{R^d}\int_{Y_{R-R^{\delta}}}\hspace{-0.85cm}e_k\cdot A\nabla({w}^l_{T}-{w}^{R,l}_{T})~dx \right|+\left|\frac{1}{R^d}\int_{Y_R\setminus Y_{R-R^{\delta}}}\hspace{-0.95cm}e_k\cdot A\nabla{w}^l_{T}~dx \right|+\left|\frac{1}{R^d}\int_{Y_R\setminus Y_{R-R^{\delta}}}\hspace{-0.95cm}e_k\cdot A\nabla{w}^{R,l}_{T}~dx \right|\\
	&\lesssim\left|\frac{1}{R^d}\int_{Y_{R-R^{\delta}}}\hspace{-0.85cm}e_k\cdot A\nabla({w}^l_{T}-{w}^{R,l}_{T})~dx \right|+\left|\int_{Y_1\setminus Y_{1-R^{\delta-1}}}\hspace{-0.85cm}e_k\cdot A\nabla\tilde{w}^l_{T}~dx \right|+\left|\int_{Y_1\setminus Y_{1-R^{\delta-1}}}\hspace{-0.85cm}e_k\cdot A\nabla\tilde{w}^{R,l}_{T}~dx \right|\\
	&\lesssim\left(\frac{1}{R^d}\int_{Y_{R-R^{\delta}}}\hspace{-0.5cm} |\nabla({w}^l_{T}-{w}^{R,l}_{T})|^2~dx \right)^{1/2}+\left(\int_{Y_1\setminus Y_{1-R^{\delta-1}}}\hspace{-0.5cm}|\nabla\tilde{w}^l_{T}|~dx \right)+\left(\int_{Y_1\setminus Y_{1-R^{\delta-1}}}\hspace{-0.5cm}|\nabla\tilde{w}^{R,l}_{T}|~dx \right)\\
	&\lesssim\left(\frac{1}{R^d}\int_{Y_{R-R^{\delta}}} |\nabla({w}^l_{T}-{w}^{R,l}_{T})|^2~dx \right)^{1/2}+\left(\int_{Y_1\setminus Y_{1-R^{\delta-1}}} |\nabla\tilde{w}^l_{T}|^2~dx\int_{Y_1\setminus Y_{1-R^{\delta-1}}} 1~dx \right)^{1/2}\\
	&\qquad\qquad+\left(\int_{Y_1\setminus Y_{1-R^{\delta-1}}} |\nabla\tilde{w}^{R,l}_{T}|^2~dx\int_{Y_1\setminus Y_{1-R^{\delta-1}}} 1~dx \right)^{1/2}\\
	&\lesssim R^{d-3\delta+d\delta}\exp\left(-c\frac{R^\delta}{\sqrt{T}}\right)+R^{(\delta-1)/2}.
	\end{align*}
	\endgroup
\end{proof}

\subsection{Rate of convergence of regularized Periodic correctors}\label{lastpart}
In this subsection, the convergence rate for the last term in~\eqref{fourparts} is established.

\begin{theorem}\label{convergence4} Let $\rho(A,L)$ satisfy $\rho(A,L)\lesssim 1/L^\tau$ for some $\tau>0$. Then for any $0<\gamma<\frac{\tau}{\tau+1}$, 
	\begin{align}|A^{R,*}-A^{R,*}_{T}|\leq C_\gamma R^{4-2\gamma} T^{-2}.\end{align}
\end{theorem}

\begin{proof}
	Observe that 
	\begin{align*}
	\xi\cdot A^{R,*}\xi=\frac{1}{|Y_R|}\int_{Y_R}(\xi+\nabla w^{R,\xi})\cdot A(\xi+\nabla w^{R,\xi})~dy,
	\end{align*} and
	\begin{align*}
	\xi\cdot A^{R,*}_{T}\xi=\frac{1}{|Y_R|}\int_{Y_R}(\xi+\nabla w^{R,\xi}_{T})\cdot A(\xi+\nabla w^{R,\xi}_{T})~dy,
	\end{align*} where $w^{R,\xi}$ solves~\eqref{cell4} and $w^{R,\xi}_{T}$ solves~\eqref{cell3}. Hence,
	\begin{align*}
	\xi\cdot&(A^{R,*}_{T}-A^{R,*})\xi
	\\&=\dashint_{Y_R}(\xi+\nabla w^{R,\xi}_{T})\cdot A(\xi+\nabla w^{R,\xi}_{T})-(\xi+\nabla w^{R,\xi})\cdot A(\xi+\nabla w^{R,\xi})~dy\\
	&=\dashint_{Y_R}(\xi+\nabla w^{R,\xi}_{T})\cdot A\nabla(w^{R,\xi}_{T}-w^{R,\xi})+\nabla(w^{R,\xi}_{T}- w^{R,\xi})\cdot A(\xi+\nabla w^{R,\xi})~dy\\
	&=\dashint_{Y_R}(\xi+\nabla w^{R,\xi}_{T})\cdot A\nabla(w^{R,\xi}_{T}-w^{R,\xi})-\nabla(w^{R,\xi}_{T}- w^{R,\xi})\cdot A(\xi+\nabla w^{R,\xi})~dy\\
	&=\dashint_{Y_R}\nabla(w^{R,\xi}_{T}-w^{R,\xi})\cdot A\nabla(w^{R,\xi}_{T}-w^{R,\xi})~dy.
	\end{align*}
	Define $\psi^R_T=-T(w^{R,\xi}_{T}-w^{R,\xi})$, then the above identity becomes
	\begin{align}\label{intermediatehomog}
	\xi\cdot(A^{R,*}_{T}-A^{R,*})\xi=T^{-2}\dashint_{Y_R}\nabla\psi^R_T\cdot A\nabla\psi^R_T~dy.
	\end{align}
	We know that $\psi^R_T\in H^1_\sharp(Y_R)$ solves the equation
	\begin{align*}
	T^{-1}\psi^R_T-\nabla\cdot(A\nabla\psi^R_T)=w^{R,\xi}\mbox{ in }Y_R.
	\end{align*} Therefore, integrating this equation against $\psi^R_T$ gives
	\begin{align*}
	T^{-1}\int_{Y_R}{|\psi^R_T|}^2~dy+\int_{Y_R}\nabla\psi^R_T\cdot A\nabla\psi^R_T~dy=\int_{Y_R}w^{R,\xi}\psi^R_T~dy.
	\end{align*}
	Dropping the first term on LHS yields
	\begin{align*}
	\int_{Y_R}\nabla\psi^R_T\cdot A\nabla\psi^R_T~dy\leq\int_{Y_R}w^{R,\xi}\psi^R_T~dy.
	\end{align*}
	Hence, 
	\begin{align*}
	\int_{Y_R}\nabla\psi^R_T\cdot A\nabla\psi^R_T~dy\leq||w^{R,\xi}||_{L^2(Y_R)}||\psi^R_T||_{L^2(Y_R)}.
	\end{align*}
	By coercivity of $A$,
	\begin{align*}
	\alpha\int_{Y_R}|\nabla\psi^R_T|^2~dy\leq||w^{R,\xi}||_{L^2(Y_R)}||\psi^R_T||_{L^2(Y_R)}.
	\end{align*}
	On applying Poincar\'e inequality:
	\begin{align*}
	\alpha||\nabla\psi^R_T||^2_{L^2(Y_R)}\lesssim R||w^{R,\xi}||_{L^2(Y_R)}||\nabla\psi^R_T||_{L^2(Y_R)},
	\end{align*}or
	\begin{align*}
	||\nabla\psi^R_T||_{L^2(Y_R)}\lesssim R||w^{R,\xi}||_{L^2(Y_R)}.
	\end{align*}
	Substituting the above in~\eqref{intermediatehomog} gives
	\begin{align}\label{periodicstep}
	\xi\cdot(A^{R,*}_{T}-A^{R,*})\xi\lesssim R^2T^{-2}\dashint_{Y_R}|w^{R,\xi}|^2~dy.
	\end{align}
	For $x\in Y_1$, define $\tilde{w}^{R,\xi}(x)=\frac{1}{R}w^{R,\xi}(Rx)$, then $\tilde{w}^{R,\xi}$ satisfies the equation:
	\begin{align}\label{rescaledapprox}
	-\nabla\cdot(A(Rx)(\xi+\nabla\tilde{w}^{R,\xi}(x))=0,&~x\in Y_1\\
	\tilde{w}^{R,\xi}(x)\mbox{ is }Y_1-\mbox{periodic}.
	\end{align}
	This equation is a particular case of the following homogenization problem:
	\begin{align}\label{arbitrary}
	-\nabla\cdot A\left(\frac{x}{\epsilon}\right)\left(z+\nabla v^\epsilon\right)=h & \mbox{ in }\tilde{\Omega},
	\end{align} where $z\in L^2(\tilde{\Omega})$, $h\in H^{-1}(\tilde{\Omega})$. By~\cite[Theorem~5.2]{Jikov1994}, if the solutions $v^\epsilon$ converge weakly to $v^0$ in $H^1_0(\tilde{\Omega})$, then $v^0$ satisfies the equation     
	\begin{align*}
	-\nabla\cdot A^*\left(z+\nabla v^0\right)=h,~x\in\tilde{\Omega}.
	\end{align*} Therefore, $\tilde{w}^{R,\xi}\rightharpoonup \tilde{w}^{\infty}$ in $H^1_\sharp(Y_1)$, which satisfies the equation
	\begin{align*}
	-\nabla\cdot A^*(\xi+\nabla\tilde{w}^{\infty})=0,~x\in Y_1.
	\end{align*} The zero mean condition on $\tilde{w}^{\infty}$ forces $\tilde{w}^{\infty}=0$ a.e.
	Now, by a similar analysis to~\cite{Shen2015}, we can obtain a rate of convergence estimate for the strong $L^2$-convergence of $w^{R,\xi}\to 0$ in the following form:
	\begin{align*}
	||\tilde{w}^{R,\xi}||_{L^2(Y_1)}=||\tilde{w}^{R,\xi}-\tilde{w}^{\infty}||_{L^2(Y_1)}\leq C_\gamma R^{-\gamma},
	\end{align*} for any $0<\gamma<\frac{\tau}{\tau+1}$.
	Finally, it follows that
	\begin{align}\label{almostperiodicstep}
	\xi\cdot(A^{R,*}_{T}-A^{R,*})\xi&\lesssim R^4T^{-2}\dashint_{Y_R}|{w}^{R,\xi}(x)|^2~dx\nonumber\\
	&\lesssim R^4T^{-2}\int_{Y_1}|\tilde{w}^{R,\xi}(x)|^2~dx\nonumber\\
	&\lesssim R^{4-2\gamma}T^{-2}.
	\end{align}
\end{proof}

\subsection{Proof of Theorem~\ref{THEOREM}}

\begin{proof}
	Using theorems~\ref{convergence1},~\ref{convergence2},~\ref{convergence3},~\ref{convergence4} and the inequality~\eqref{fourparts}, we obtain
	\begin{align}
	|A^*-A^{R,*}|\lesssim \frac{1}{{T}^{\frac{\tau}{2(\tau+1)}-\omega}}+\frac{1}{L^\tau}+\left(\frac{L}{R}\right)^{1/2}+R^{d}\exp\left(-c\frac{R^\delta}{\sqrt{T}}\right)+\frac{1}{R^{(1-\delta)/2}}+ R^{4-2\gamma} T^{-2}.
	\end{align}
	Let $\gamma{'},\beta_1\in(0,1)$. By choosing $\gamma{'}=\gamma/2$, $T=R^{2-\gamma{'}}$, $L=R^{\beta_1}$, $\beta_2=2\gamma{'}$ and $\delta=1-\beta_2/8$, we can obtain the estimate $|A^*-A^{R,*}|\lesssim \frac{1}{R^{\beta}}$, for some $\beta>0$.
\end{proof}

\begin{remark}
	\leavevmode
	\begin{enumerate}
		\item The above theorem quantifies the rate of convergence of approximate homogenized tensor corresponding to periodizations for a class of almost periodic media. Interestingly, this approximation can also be computed through Bloch wave methods as described earlier. Indeed, the approximation corresponds to half the Hessian of the first Bloch eigenvalue of the periodized operator. In the Physics literature, Bloch wave tools are used even for non-periodic media. This result provides quantitative justification for the use of Bloch wave method for non-periodic media.
		\item The approximate cell problem with periodic boundary conditions~\eqref{CellP} is not unique. Bourgeat-Piatnitski~\cite{BourgeatPiatnitski2004} also prove rate of convergence estimates for Dirichlet and Neumann approximations. Our theorem can be modified appropriately for Dirichlet and Neumann approximations. In the place of rate of convergence estimates for almost periodic homogenization of periodic BVPs, which were used in Subsection~\ref{lastpart}, one would require rate of convergence estimates for almost periodic homogenization of Dirichlet and Neumann BVPs. These estimates are available in~\cite{Shen2015,ArmstrongShen2016}. 
	\end{enumerate}
\end{remark}

\section{Numerical study}\label{numerics} In this section, we report on the numerical experiments that we carried out for certain benchmark periodic and quasiperiodic functions introduced in~\cite{Gloria2011,GloriaHabibi2016}. It is known that approximations of homogenized tensor for periodic media using Dirichlet and Periodic correctors have a rate of convergence of $R^{-1}$~\cite[Cor.~1]{Abdulle2019}. Our aim is to verify such results. We also numerically study the difference of Dirichlet and Periodic correctors as we feel that this difference should also show decay. These computations are done using the finite element method on FEniCS software~\cite{AlnaesBlechta2015a}.

\subsection{Numerical study for approximations of $A^*$ using periodic correctors}
In this subsection, we investigate the behavior of the error in approximations to homogenized tensor $|A^{R,*}-A^*|$ using periodic correctors with respect to $R$.  The approximate homogenized tensor $A^{R,*}$ corresponding to periodization $A^R$ has already been defined in~\eqref{homoR}.

The first two examples are that of periodic matrices \begin{align*}A_1(x)=\left(\frac{2+1.8 \sin(2\pi x)}{2+1.8 \cos(2\pi y)}+\frac{2+\sin(2\pi y)}{2+1.8 \cos(2\pi x)}\right)\Id,\mbox{ and }\end{align*} \begin{align*}A_2(x)=(1+30(2+\sin(2\pi x)\sin(2\pi y)))\Id.\end{align*} The homogenized tensor $A^*$ is computed numerically by solving the periodic cell problem on the unit cube $[0,1)^d$ and is found to be approximately $2.757\Id$ and $59.1\Id$ for $A_1$ and $A_2$ respectively.

The third example is that of the following matrix with quasiperiodic entries:
\begin{align*}A_3(x)=
\left(4+\cos(2\pi(x+y))+\cos(2\pi\sqrt{2}(x+y))\right)\Id\end{align*}
The homogenized coefficient for quasiperiodic media $A^*$~\eqref{homoAP} is defined as a  mean value in the full space $\mathbb{R}^d$ and therefore it is impossible to compute. Hence, for the computation of the error, $A^*$ is taken to be the approximate homogenized tensor
\begin{align*}\xi\cdot A^{R,D,*}_{T}\xi=\frac{1}{|Y_R|}\int_{Y_R}\left((\xi+\nabla w_{T}^{R,D,\xi})\cdot A(\xi+\nabla w_{T}^{R,D,\xi})\right)~dy.\end{align*}
corresponding to the following cell problem:
Find $ w_{T}^{R,D,\xi}\in H^1_0(Y_R)$ such that 
\begin{align*}
-\nabla\cdot(A (\xi+\nabla w_{T}^{R,D,\xi}))+T^{-1}w_{T}^{R,D,\xi}=0.
\end{align*} for $R=T=60$, since $A^{R,D,*}_T$ is known to converge faster to $A^*$ as $R,T\to\infty$~\cite{Gloria2011,GloriaHabibi2016}. 

The log-log plots of errors in periodic and quasiperiodic cases seem to suggest an asymptotically polynomial rate of convergence. Computations are performed with P1 finite elements with a varying choice of number of meshpoints $n$ per dimension, as denoted in Figures~\ref{fig:logPeriodic} and~\ref{fig:logerrorplotfuncAPP}.

\begin{figure}
	\begin{subfigure}[b]{.495\linewidth}
		\centering
		\includegraphics[width=\linewidth]{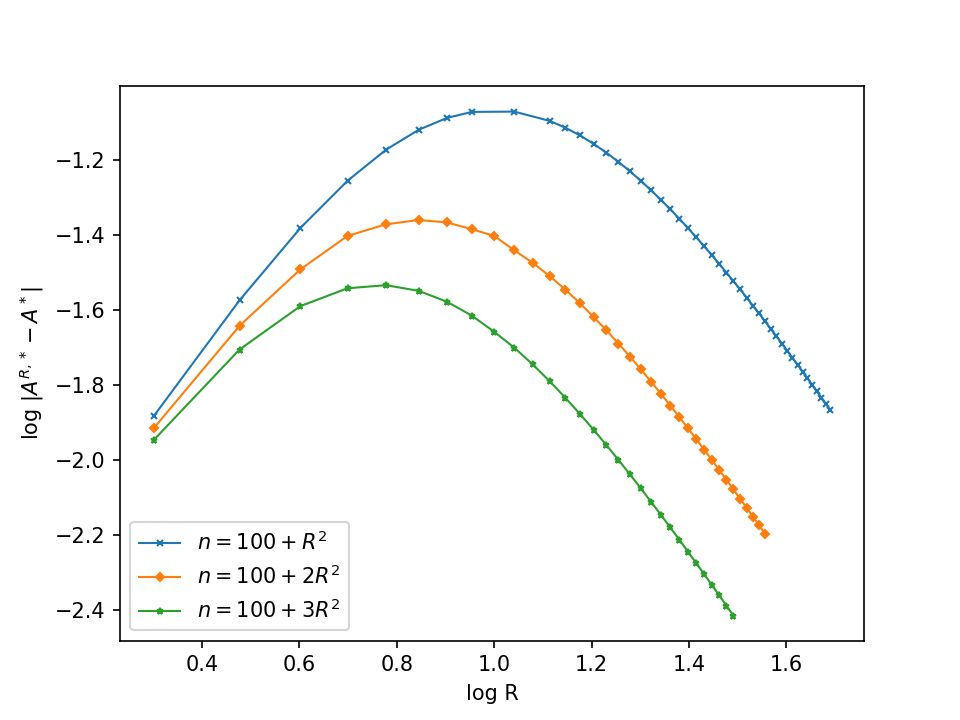}
		\caption{Periodic function $A_1$}
		\label{fig:P1}
	\end{subfigure}
	\begin{subfigure}[b]{.495\linewidth}
		\centering
		\includegraphics[width=\linewidth]{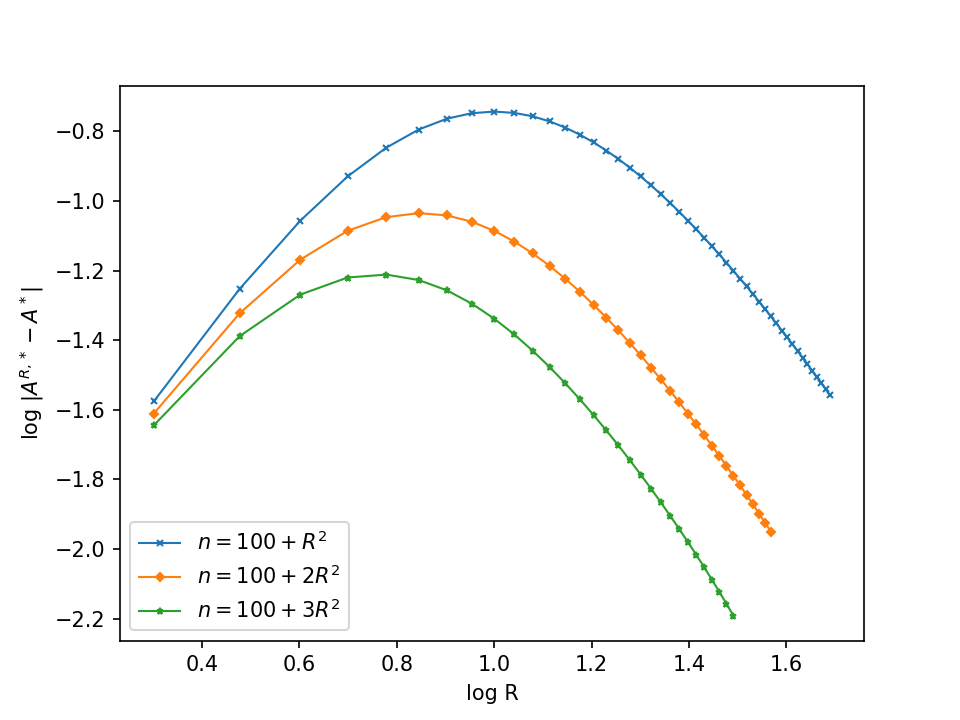}
		\caption{Periodic Function $A_2$}
		\label{fig:P2}
	\end{subfigure}
	\caption{The error $|A^{R,*}-A^*|$ for approximations to homogenized tensor using periodic correctors in log-log scale for the functions $A_1$ and $A_2$ with respect to $R$}
	\label{fig:logPeriodic}
\end{figure}

\begin{figure}
	\centering
	\includegraphics[width=0.8\linewidth]{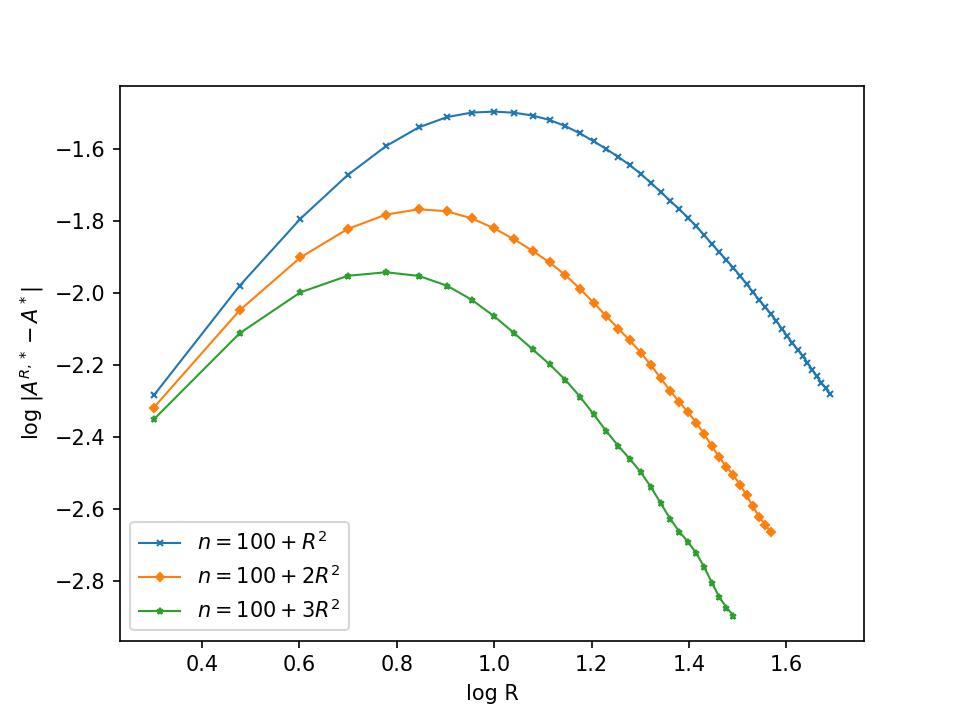}
	\caption{The error $|A^{R,*}-A^*|$ for approximations to homogenized tensor using periodic correctors in log-log scale for the function $A_3$ with respect to $R$}
	\label{fig:logerrorplotfuncAPP}
\end{figure}

\subsection{Numerical study for Dirichlet Approximations}
The cell problem for almost periodic media~\eqref{cellAP} is posed in $\mathbb{R}^d$. The following is its Dirichlet approximation, which is the truncation of~\eqref{cellAP} on a cube $Y_R=[-R\pi,R\pi)^d$ of side length $2\pi R$. Let $H^1_0(Y_R)$ denote the space of all $L^2(Y_R)$ functions whose weak derivatives are also in $L^2(Y_R)$ and whose trace on $Y_R$ is zero. 

Given $\xi\in\mathbb{R}^d$, find $w^{R, D,\xi}\in H^1_0(Y_R)$ such that
\begin{align}\label{DirichletApproximation}
-\nabla\cdot A(\xi+\nabla w^{R,D,\xi})=0.
\end{align} 
Then Dirichlet approximation $A^{R,D,*}=\left(a^{R,D,*}_{kl}\right)$ to the homogenized tensor is given by 
\begin{align}\label{homoD}
a^{R,D,*}_{kl}=\mathcal{M}_{Y_R}\left(a_{kl}+\sum_{j=1}^{d}a_{kj}\frac{\partial w^{R,D,e_l}}{\partial y_j}\right).
\end{align}

In this subsection, we investigate the behavior of the error in the Dirichlet approximations $|A^{R,D,*}-A^*|$ with respect to side length $R$. 
The approximate homogenized tensor $A^{R,D,*}$ is computed by solving the Dirichlet cell problem~\eqref{cell4} for different values of $R$ going up to $40$. The computations are carried out with P2-Finite Elements discretization and $20$ points per dimension in every unit cell.  See Figure~\ref{fig:logerrorplotfunctionDirichlet} for the log-log plot of the error $|A^{R,D,*}-A^*|$ with respect to $R$ for matrices $A_1$ and $A_2$.

\begin{figure}
	\begin{subfigure}[b]{.495\linewidth}
		\centering
		\includegraphics[width=\linewidth]{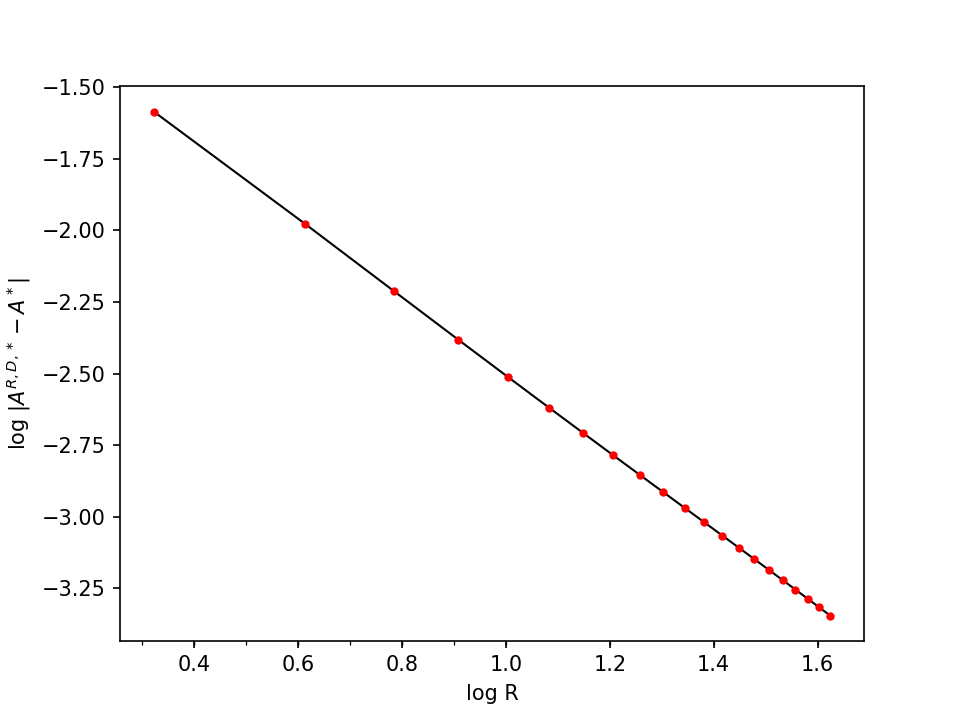}
		\caption{Periodic function  $A_1$}
		\label{fig:D1}
	\end{subfigure}
	\begin{subfigure}[b]{.495\linewidth}
		\centering
		\includegraphics[width=\linewidth]{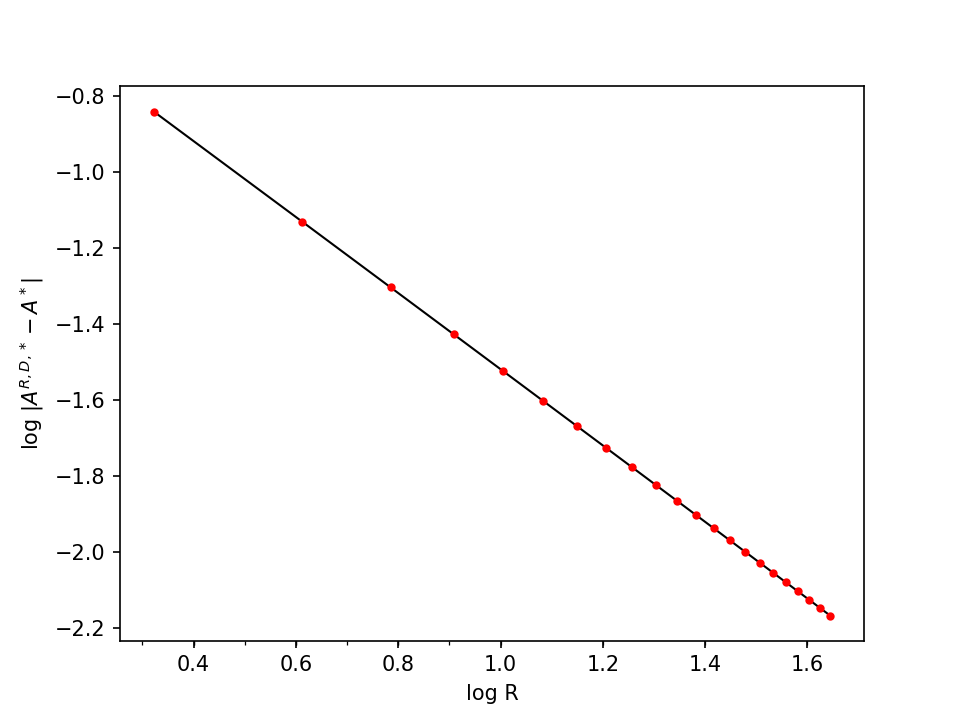}
		\caption{Periodic function $A_2$}
		\label{fig:D2}
	\end{subfigure}
\caption{The error $|A^{R,D,*}-A^*|$ for Dirichlet approximations in log-log scale for the functions $A_1$ and $A_2$ with respect to $R$.}
\label{fig:logerrorplotfunctionDirichlet}
\end{figure}

See Figure~\ref{fig:logerrorplotapfunc} for the log-log plot of the error $|A^{R,D,*}-A^*|$ with respect to $R$ for the matrix $A_3$.
\begin{figure}
	\centering
	\includegraphics[width=0.789\linewidth]{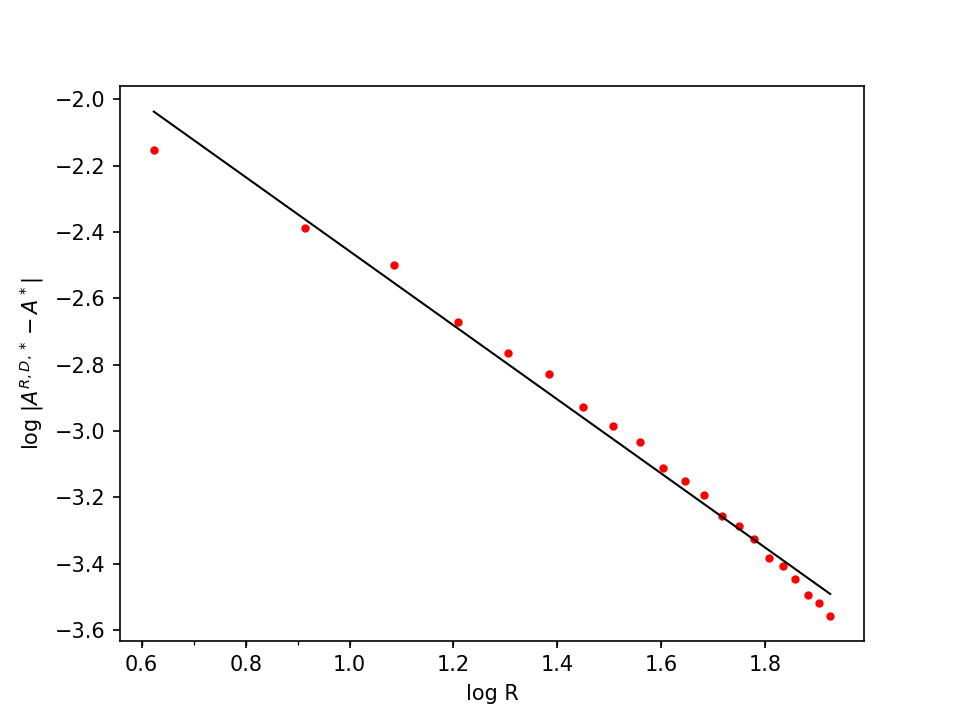}
	\caption{The error $|A^{R,D,*}-A^*|$ for Dirichlet approximations in log-log scale for the function $A_3$ with respect to $R$.}
	\label{fig:logerrorplotapfunc}
\end{figure}

\subsection{Comparison of Dirichlet and Periodic Correctors}
An interesting question that arises is whether the Dirichlet and periodic correctors, respectively $w^{R,D,\xi}$ and $w^{R,\xi}$, grow close to each other as the side length $R$ of sample cube increases. 

In Figure~\ref{fig:logDvP}, we plot the error $E(R)=\left(\dashint_{Y_{R}}|\nabla w^{R,D,e_1}(y)-\nabla w^{R,e_1}(y)|^2\,dy\right)^{1/2}$ with respect to $R$ on a log-log scale for functions $A_1$ and $A_2$. In Figure~\ref{fig:logerrorplotfuncDVPAP}, we plot the error $E(R)$ with respect to $R$ on a log-log scale for $A_3$. 
\begin{figure}
	\begin{subfigure}[b]{.495\linewidth}
		\centering
		\includegraphics[width=\linewidth]{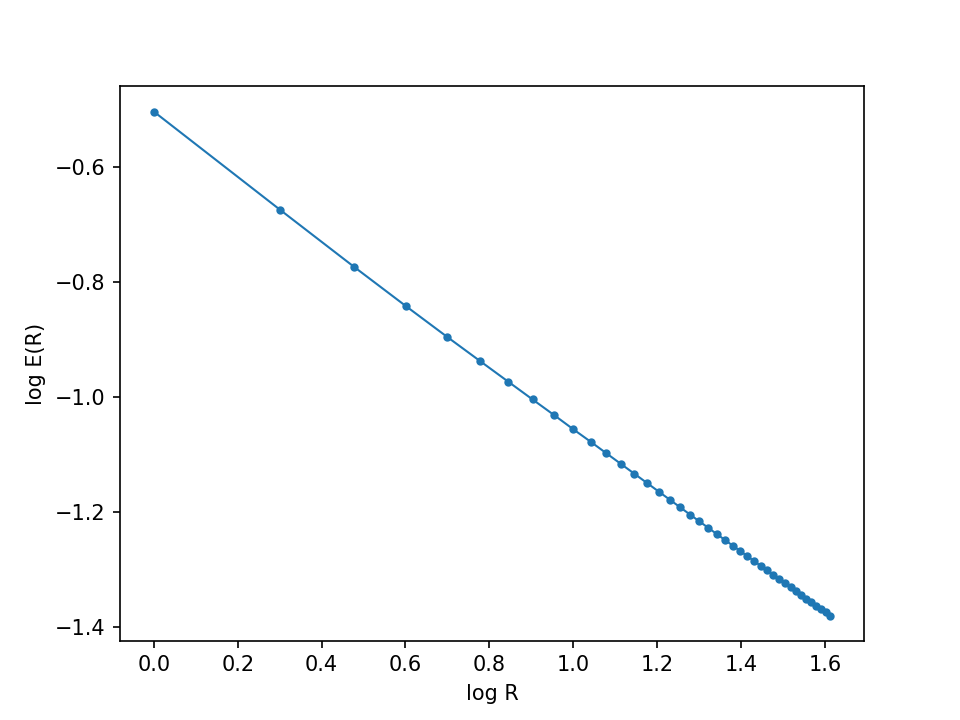}
		\caption{Periodic function $A_1$}
		\label{fig:DVP1}
	\end{subfigure}
	\begin{subfigure}[b]{.495\linewidth}
		\centering
	    \includegraphics[width=\linewidth]{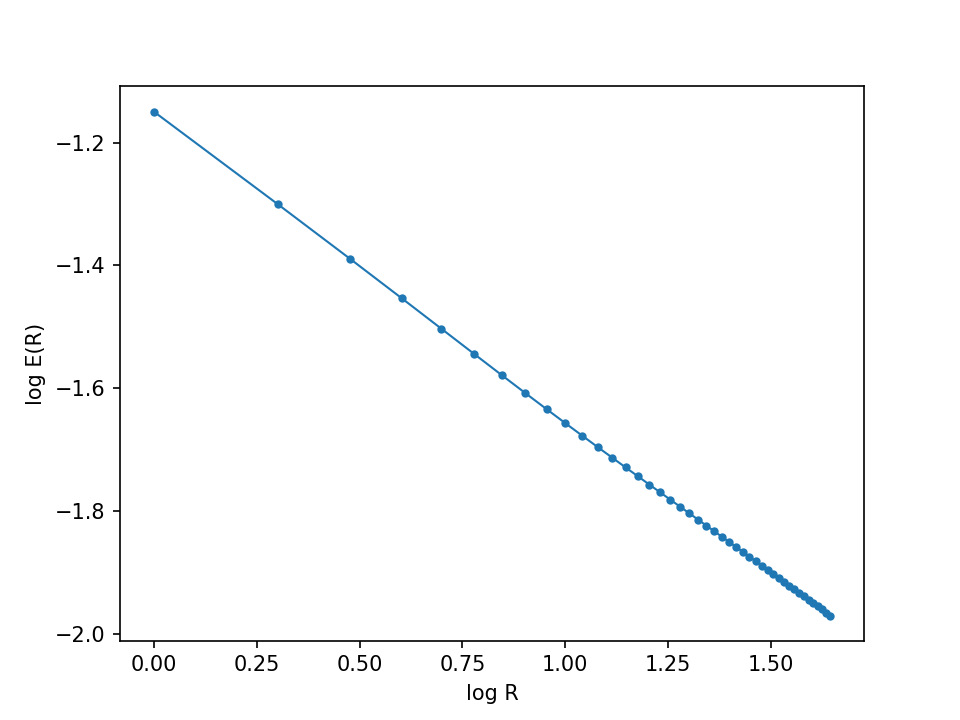}
		\caption{Periodic function $A_2$}
		\label{fig:DVP2}
	\end{subfigure}
	\caption{The averaged $L^2$ norm of the difference of the gradients $E(R)=\left(\dashint_{Y_{R}}|\nabla w^{R,D,e_1}(y)-\nabla w^{R,e_1}(y)|^2\,dy\right)^{1/2}$  in log-log scale for the correctors corresponding to the periodic matrices $A_1$ and $A_2$ plotted as a function of $R$.}
\label{fig:logDvP}
	\centering
	\includegraphics[width=0.85\linewidth]{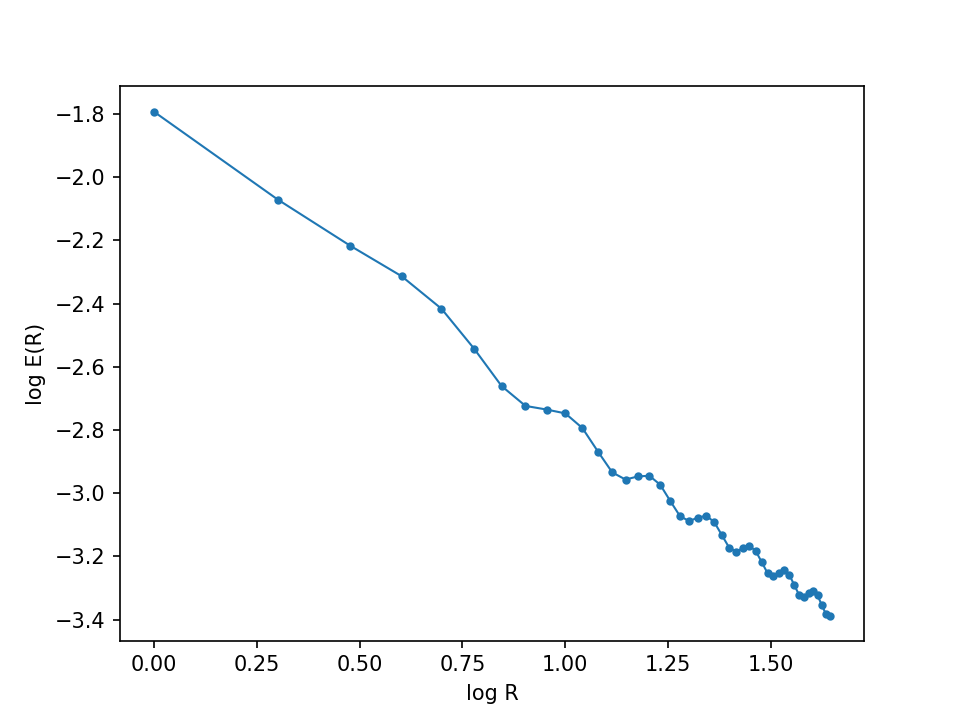}
	\caption{The averaged $L^2$ norm of the difference of the gradients $E(R)=\left(\dashint_{Y_{R}}|\nabla w^{R,D,e_1}(y)-\nabla w^{R,e_1}(y)|^2\,dy\right)^{1/2}$  in log-log scale for the correctors corresponding to the quasiperiodic matrix $A_3$ plotted as a function of $R$.}
	\label{fig:logerrorplotfuncDVPAP}
\end{figure}
The numerical study is carried out with P1 finite elements. The number of meshpoints per dimension is taken to be $n=100+R^2$.

In Figure~\ref{fig:logDvPhom}, we plot the error $|A^{R,D,*}-A^{R,*}|$ with respect to $R$ on a log-log scale for functions $A_1$ and $A_2$. In Figure~\ref{fig:logerrorplotfuncDVPAPhom}, we plot the error $|A^{R,D,*}-A^{R,*}|$ with respect to $R$ on a log-log scale for $A_3$. 

\begin{figure}
	\begin{subfigure}[b]{.495\linewidth}
		\centering
		\includegraphics[width=\linewidth]{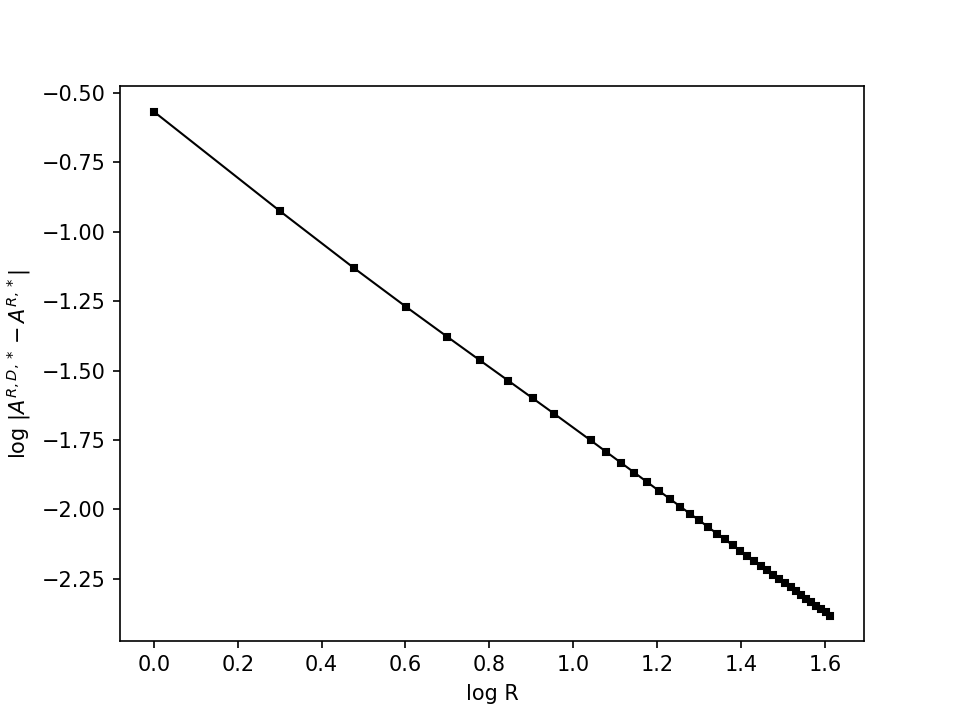}
		\caption{Periodic function $A_1$}
		\label{fig:DVPhom1}
	\end{subfigure}
	\begin{subfigure}[b]{.495\linewidth}
		\centering
		\includegraphics[width=\linewidth]{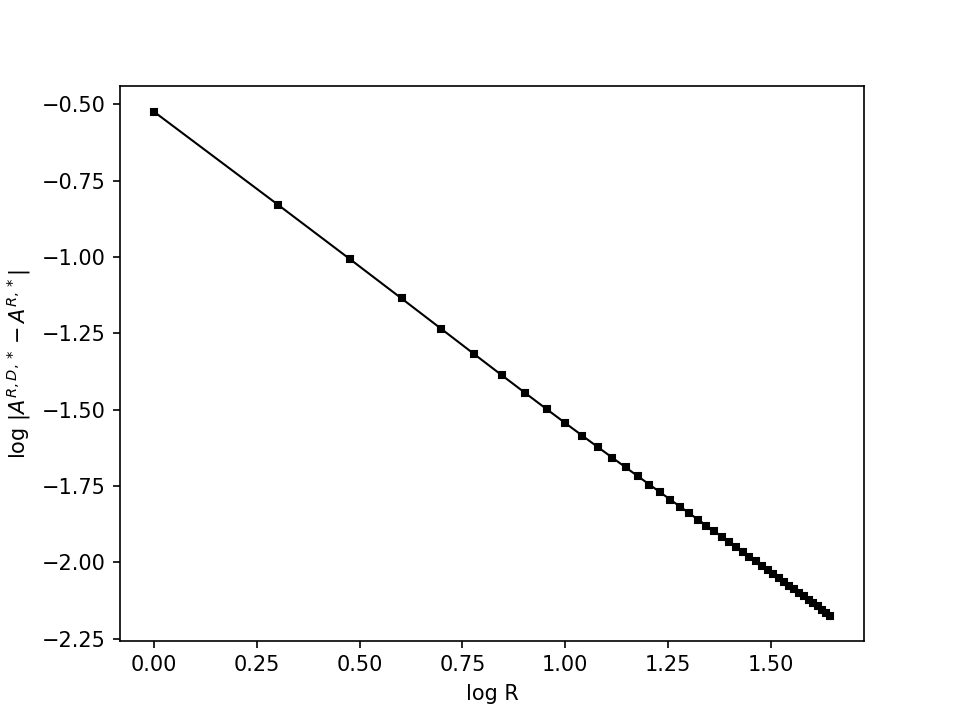}
		\caption{Periodic function $A_2$}
		\label{fig:DVPhom2}
	\end{subfigure}
	\caption{The absolute error $|A^{R,D,*}-A^{R,*}|$  in log-log scale for the periodic matrices $A_1$ and $A_2$ plotted as a function of $R$.}
	\label{fig:logDvPhom}
	\centering
	\includegraphics[width=0.85\linewidth]{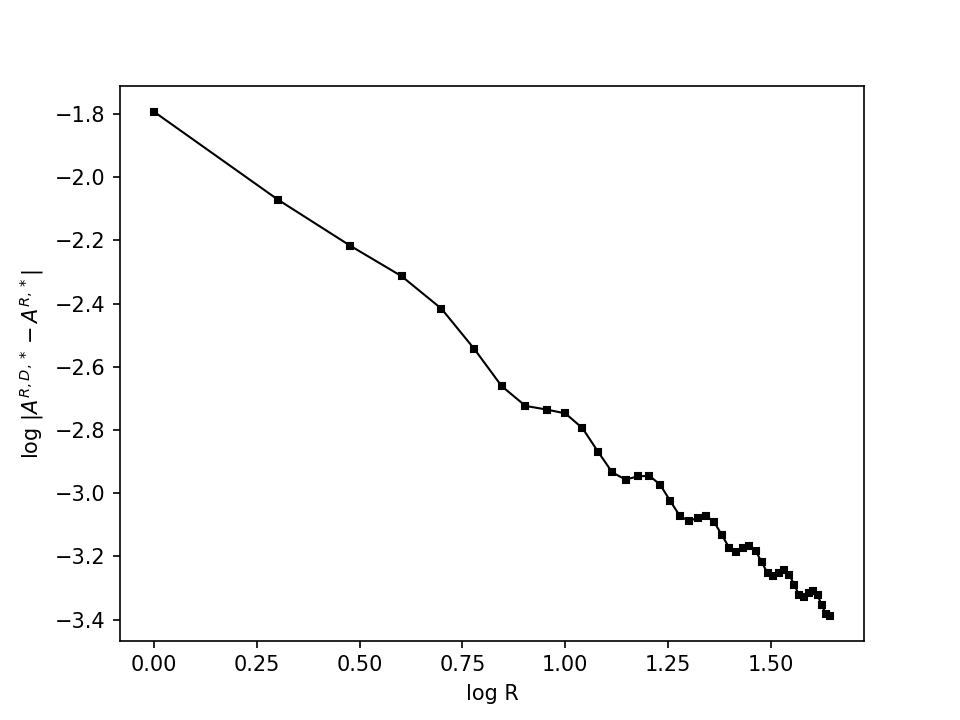}
	\caption{The absolute error $|A^{R,D,*}-A^{R,*}|$  in log-log scale for the quasiperiodic matrix $A_3$ plotted as a function of $R$.}
	\label{fig:logerrorplotfuncDVPAPhom}
\end{figure}

\section*{Acknowledgements}
We would like to thank Prof. Harsha Hutridurga for pointing us to~\cite{BourgeatPiatnitski2004}. We acknowledge SpaceTime-2 supercomputing facility at IIT Bombay for the computing time.

\bibliographystyle{apalike}
\bibliography{mylit}

\def\cprime{$'$}
\begin{thebibliography}{}

\bibitem[Abdulle et~al., 2019]{Abdulle2019}
Abdulle, A., Arjmand, D., and Paganoni, E. (2019).
\newblock Exponential decay of the resonance error in numerical homogenization
  via parabolic and elliptic cell problems.
\newblock {\em C. R. Math. Acad. Sci. Paris}, 357(6):545--551.

\bibitem[Allaire et~al., 2004]{AllaireVanni2004}
Allaire, G., Capdeboscq, Y., Piatnitski, A., Siess, V., and Vanninathan, M.
  (2004).
\newblock Homogenization of periodic systems with large potentials.
\newblock {\em Arch. Ration. Mech. Anal.}, 174(2):179--220.

\bibitem[Allaire et~al., 2011]{AllaireRauch2011}
Allaire, G., Palombaro, M., and Rauch, J. (2011).
\newblock Diffractive geometric optics for {B}loch wave packets.
\newblock {\em Arch. Ration. Mech. Anal.}, 202(2):373--426.

\bibitem[Allaire and Piatnitski, 2005]{AllairePiatnitski2005}
Allaire, G. and Piatnitski, A. (2005).
\newblock Homogenization of the {S}chr\"odinger equation and effective mass
  theorems.
\newblock {\em Comm. Math. Phys.}, 258(1):1--22.

\bibitem[Allais, 1983]{Allais1983}
Allais, M. (1983).
\newblock Sur la distribution normale des valeurs \`a des instants
  r\'{e}guli\`erement espac\'{e}s d'une somme de sinuso\"{\i}des.
\newblock {\em C. R. Acad. Sci. Paris S\'{e}r. I Math.}, 296(19):829--832.

\bibitem[Aln{\ae}s et~al., 2015]{AlnaesBlechta2015a}
Aln{\ae}s, M.~S., Blechta, J., Hake, J., Johansson, A., Kehlet, B., Logg, A.,
  Richardson, C., Ring, J., Rognes, M.~E., and Wells, G.~N. (2015).
\newblock The fenics project version 1.5.
\newblock {\em Archive of Numerical Software}, 3(100).

\bibitem[Amerio and Prouse, 1971]{Amerio1971}
Amerio, L. and Prouse, G. (1971).
\newblock {\em Almost-periodic functions and functional equations}.
\newblock Van Nostrand Reinhold Co., New York-Toronto, Ont.-Melbourne.

\bibitem[Armstrong et~al., 2016]{ArmstrongGloriaKuusi2016}
Armstrong, S., Gloria, A., and Kuusi, T. (2016).
\newblock Bounded correctors in almost periodic homogenization.
\newblock {\em Arch. Ration. Mech. Anal.}, 222(1):393--426.

\bibitem[Armstrong et~al., 2014]{Armstrong2014}
Armstrong, S.~N., Cardaliaguet, P., and Souganidis, P.~E. (2014).
\newblock Error estimates and convergence rates for the stochastic
  homogenization of {H}amilton-{J}acobi equations.
\newblock {\em J. Amer. Math. Soc.}, 27(2):479--540.

\bibitem[Armstrong and Shen, 2016]{ArmstrongShen2016}
Armstrong, S.~N. and Shen, Z. (2016).
\newblock Lipschitz estimates in almost-periodic homogenization.
\newblock {\em Comm. Pure Appl. Math.}, 69(10):1882--1923.

\bibitem[Avellaneda and Lin, 1987]{AvellanedaLin1987}
Avellaneda, M. and Lin, F.-H. (1987).
\newblock Compactness methods in the theory of homogenization.
\newblock {\em Comm. Pure Appl. Math.}, 40(6):803--847.

\bibitem[Bellissard and Testard, 1981]{bellissard1981almost}
Bellissard, J. and Testard, D. (1981).
\newblock Almost periodic hamiltonians: an algebraic approach.
\newblock Technical report, Centre National de la Recherche Scientifique.

\bibitem[Benoit and Gloria, 2017]{Gloria2017}
Benoit, A. and Gloria, A. (2017).
\newblock Long-time homogenization and asymptotic ballistic transport of
  classical waves.
\newblock \url{https://arXiv.org/abs/1701.08600}.

\bibitem[Bensoussan et~al., 2011]{Bensoussan2011}
Bensoussan, A., Lions, J.-L., and Papanicolaou, G. (2011).
\newblock {\em Asymptotic analysis for periodic structures}.
\newblock AMS Chelsea Publishing, Providence, RI.

\bibitem[Besicovitch, 1955]{Besicovitch55}
Besicovitch, A.~S. (1955).
\newblock {\em Almost periodic functions}.
\newblock Dover Publications, Inc., New York.

\bibitem[Blanc and Le~Bris, 2010]{BlancLeBris2010}
Blanc, X. and Le~Bris, C. (2010).
\newblock Improving on computation of homogenized coefficients in the periodic
  and quasi-periodic settings.
\newblock {\em Netw. Heterog. Media}, 5(1):1--29.

\bibitem[Blanc et~al., 2015]{Blanc2015}
Blanc, X., Le~Bris, C., and Lions, P.-L. (2015).
\newblock Local profiles for elliptic problems at different scales: defects in,
  and interfaces between periodic structures.
\newblock {\em Comm. Partial Differential Equations}, 40(12):2173--2236.

\bibitem[Bourgeat and Piatnitski, 2004]{BourgeatPiatnitski2004}
Bourgeat, A. and Piatnitski, A. (2004).
\newblock Approximations of effective coefficients in stochastic
  homogenization.
\newblock {\em Ann. Inst. H. Poincar\'{e} Probab. Statist.}, 40(2):153--165.

\bibitem[Bourgeat et~al., 1988]{Whitaker1988}
Bourgeat, A., Quintard, M., and Whitaker, S. (1988).
\newblock \'{E}l\'{e}ments de comparaison entre la m\'{e}thode
  d'homog\'{e}n\'{e}isation et la m\'{e}thode de prise de moyenne avec
  fermeture.
\newblock {\em C. R. Acad. Sci. Paris S\'{e}r. II M\'{e}c. Phys. Chim. Sci.
  Univers Sci. Terre}, 306(7):463--466.

\bibitem[B\u{a}lilescu et~al., 2018]{Ghosh2018}
B\u{a}lilescu, L., Conca, C., Ghosh, T., San~Mart\'{i}n, J., and Vanninathan,
  M. (2018).
\newblock The {D}ispersion {T}ensor and {I}ts {U}nique {M}inimizer in
  {H}ashin--{S}htrikman {M}icro-structures.
\newblock {\em Arch. Ration. Mech. Anal.}, 230(2):665--700.

\bibitem[Carvalho and de~Oliveira, 2002]{Carvalho2002}
Carvalho, T.~O. and de~Oliveira, C.~R. (2002).
\newblock Spectra and transport in almost periodic dimers.
\newblock {\em J. Statist. Phys.}, 107(5-6):1015--1030.

\bibitem[Casado-D\'{i}az and Gayte, 2002]{CasadoDiaz2002}
Casado-D\'{i}az, J. and Gayte, I. (2002).
\newblock A derivation theory for generalized {B}esicovitch spaces and its
  application for partial differential equations.
\newblock {\em Proc. Roy. Soc. Edinburgh Sect. A}, 132(2):283--315.

\bibitem[Cherkaev and Kohn, 1997]{cherkaev97}
Cherkaev, A. and Kohn, R., editors (1997).
\newblock {\em Topics in the mathematical modelling of composite materials},
  volume~31.
\newblock Birkh\"auser Boston, Inc., Boston, MA.

\bibitem[Conca and Vanninathan, 1997]{Conca1997}
Conca, C. and Vanninathan, M. (1997).
\newblock Homogenization of periodic structures via bloch decomposition.
\newblock {\em SIAM Journal on Applied Mathematics}, 57(6):1639--1659.

\bibitem[Corduneanu, 2009]{Cord2009}
Corduneanu, C. (2009).
\newblock {\em Almost periodic oscillations and waves}.
\newblock Springer, New York.

\bibitem[Damanik et~al., 2019]{Damanik2019}
Damanik, D., Fillman, J., and Gorodetski, A. (2019).
\newblock Multidimensional almost-periodic schr{\"o}dinger operators with
  cantor spectrum.
\newblock {\em Annales Henri Poincar{\'e}}, 20(4):1393--1402.

\bibitem[Davit et~al., 2013]{davit2013homogenization}
Davit, Y., Bell, C.~G., Byrne, H.~M., Chapman, L.~A., Kimpton, L.~S., Lang,
  G.~E., Leonard, K.~H., Oliver, J.~M., Pearson, N.~C., Shipley, R.~J., et~al.
  (2013).
\newblock Homogenization via formal multiscale asymptotics and volume
  averaging: How do the two techniques compare?
\newblock {\em Advances in Water Resources}, 62:178--206.

\bibitem[Fink, 1974]{Fink74}
Fink, A.~M. (1974).
\newblock {\em Almost periodic differential equations}.
\newblock Lecture Notes in Mathematics, Vol. 377. Springer-Verlag, Berlin-New
  York.

\bibitem[Gloria, 2011]{Gloria2011}
Gloria, A. (2011).
\newblock Reduction of the resonance error---{P}art 1: {A}pproximation of
  homogenized coefficients.
\newblock {\em Math. Models Methods Appl. Sci.}, 21(8):1601--1630.

\bibitem[Gloria and Habibi, 2016]{GloriaHabibi2016}
Gloria, A. and Habibi, Z. (2016).
\newblock Reduction in the resonance error in numerical homogenization {II}:
  {C}orrectors and extrapolation.
\newblock {\em Found. Comput. Math.}, 16(1):217--296.

\bibitem[Gloria and Otto, 2017]{GloriaOtto2017}
Gloria, A. and Otto, F. (2017).
\newblock Quantitative results on the corrector equation in stochastic
  homogenization.
\newblock {\em J. Eur. Math. Soc.}, 19(11):3489--3548.

\bibitem[{Hofstadter}, 1976]{Hofstadter76}
{Hofstadter}, D.~R. (1976).
\newblock {Energy levels and wave functions of Bloch electrons in rational and
  irrational magnetic fields}.
\newblock {\em Phys. Rev. B}, 14:2239--2249.

\bibitem[Jikov et~al., 1994]{Jikov1994}
Jikov, V.~V., Kozlov, S.~M., and Ole\u{i}nik, O.~A. (1994).
\newblock {\em Homogenization of differential operators and integral
  functionals}.
\newblock Springer-Verlag, Berlin.

\bibitem[Kato, 1995]{Kato1995}
Kato, T. (1995).
\newblock {\em Perturbation theory for linear operators}.
\newblock Classics in Mathematics. Springer-Verlag, Berlin.

\bibitem[Katz and Duneau, 1986]{Katz1986}
Katz, A. and Duneau, M. (1986).
\newblock Quasiperiodic patterns and icosahedral symmetry.
\newblock {\em J. Physique}, 47(2):181--196.

\bibitem[Kozlov, 1978]{Kozlov78}
Kozlov, S.~M. (1978).
\newblock Averaging of differential operators with almost periodic rapidly
  oscillating coefficients.
\newblock {\em Mat. Sb. (N.S.)}, 107(149)(2):199--217, 317.

\bibitem[Kozlov, 1979]{kozlov}
Kozlov, S.~M. (1979).
\newblock The averaging of random operators.
\newblock {\em Mat. Sb. (N.S.)}, 109(151)(2):188--202, 327.

\bibitem[Levitan and Zhikov, 1982]{Levitan82}
Levitan, B.~M. and Zhikov, V.~V. (1982).
\newblock {\em Almost periodic functions and differential equations}.
\newblock Cambridge University Press.

\bibitem[Maurin, 1968]{maurin68}
Maurin, K. (1968).
\newblock {\em General eigenfunction expansions and unitary representations of
  topological groups}.
\newblock Monografie Matematyczne, Tom 48. PWN-Polish Scientific Publishers,
  Warsaw.

\bibitem[Oleinik and Zhikov, 1982]{Oleinik1982}
Oleinik, O.~A. and Zhikov, V.~V. (1982).
\newblock On the homogenization of elliptic operators with almost-periodic
  coefficients.
\newblock {\em Rendiconti del Seminario Matematico e Fisico di Milano},
  52(1):149--166.

\bibitem[Payne and Weinberger, 1960]{Payne1960}
Payne, L.~E. and Weinberger, H.~F. (1960).
\newblock An optimal {P}oincar\'{e} inequality for convex domains.
\newblock {\em Arch. Rational Mech. Anal.}, 5:286--292 (1960).

\bibitem[Pozhidaev and Yurinski\u{\i}, 1989]{Yurinski1989}
Pozhidaev, A.~V. and Yurinski\u{\i}, V.~V. (1989).
\newblock On the error of averaging of symmetric elliptic systems.
\newblock {\em Izv. Akad. Nauk SSSR Ser. Mat.}, 53(4):851--867, 912.

\bibitem[Reed and Simon, 1978]{Reed1978}
Reed, M. and Simon, B. (1978).
\newblock {\em Methods of modern mathematical physics. {IV}. {A}nalysis of
  operators}.
\newblock Academic Press [Harcourt Brace Jovanovich, Publishers], New
  York-London.

\bibitem[Shechtman et~al., 1984]{Schechtman84}
Shechtman, D., Blech, I., Gratias, D., and Cahn, J.~W. (1984).
\newblock Metallic phase with long-range orientational order and no
  translational symmetry.
\newblock {\em Phys. Rev. Lett.}, 53:1951--1953.

\bibitem[Shen, 2015]{Shen2015}
Shen, Z. (2015).
\newblock Convergence rates and {H}\"{o}lder estimates in almost-periodic
  homogenization of elliptic systems.
\newblock {\em Anal. PDE}, 8(7):1565--1601.

\bibitem[Shen and Zhuge, 2018]{ShenZhuge2018}
Shen, Z. and Zhuge, J. (2018).
\newblock Approximate correctors and convergence rates in almost-periodic
  homogenization.
\newblock {\em J. Math. Pures Appl. (9)}, 110:187--238.

\bibitem[Shubin, 1978]{Shubin78}
Shubin, M.~A. (1978).
\newblock Almost periodic functions and partial differential operators.
\newblock {\em Russian Mathematical Surveys}, 33(2):1.

\bibitem[Simon, 1982]{Simon1982}
Simon, B. (1982).
\newblock Almost periodic {S}chr\"{o}dinger operators: a review.
\newblock {\em Adv. in Appl. Math.}, 3(4):463--490.

\bibitem[Sivaji~Ganesh and Tewary, 2019]{Vivek2019ii}
Sivaji~Ganesh, S. and Tewary, V. (2019).
\newblock Bloch wave homogenization of quasiperiodic media.
\newblock \url{https://arxiv.org/abs/1910.12724}.
\newblock Accessed: 2019-10-29.

\bibitem[Sivaji~Ganesh and Tewary, 2020]{Vivek2018}
Sivaji~Ganesh, S. and Tewary, V. (2020).
\newblock Generic simplicity of spectral edges and applications to
  homogenization.
\newblock {\em {Asymptotic Analysis}}, 116(3--4):219--248.

\bibitem[Sivaji~Ganesh and Vanninathan, 2004]{Sivaji2004}
Sivaji~Ganesh, S. and Vanninathan, M. (2004).
\newblock Bloch wave homogenization of scalar elliptic operators.
\newblock {\em Asymptot. Anal.}, 39(1):15--44.

\bibitem[Sivaji~Ganesh and Vanninathan, 2005]{SivajiGanesh2005}
Sivaji~Ganesh, S. and Vanninathan, M. (2005).
\newblock Bloch wave homogenization of linear elasticity system.
\newblock {\em ESAIM Control Optim. Calc. Var.}, 11(4):542--573.

\bibitem[Whitaker, 2013]{whitaker2013method}
Whitaker, S. (2013).
\newblock {\em The method of volume averaging}, volume~13.
\newblock Springer Science \& Business Media.

\bibitem[Yurinski\u{\i}, 1986]{Yurinski1986}
Yurinski\u{\i}, V.~V. (1986).
\newblock Averaging of symmetric diffusion in a random medium.
\newblock {\em Sibirsk. Mat. Zh.}, 27(4):167--180, 215.

\end{thebibliography}
\end{document}